\documentclass[preprint,3p]{elsarticle}
\biboptions{sort&compress}
\usepackage{enumitem}
\usepackage{soul}
\newcommand{\bogus}[1]{}
\usepackage[bookmarks=true,colorlinks=true,linkcolor=blue]{hyperref}
\hypersetup{colorlinks=true, linkcolor=blue,urlcolor=red,citecolor=siaminlinkcolor,linktoc=all}

\usepackage{amssymb}
\usepackage{amsmath}
\usepackage{amsthm}

\usepackage{colortbl}

\usepackage{lineno}

\usepackage{calc}

\newtheorem{theorem}{Theorem}

\newtheorem{definition}{Definition}

\newcommand{\das}[1]{{{#1}}}

\usepackage{subcaption}

\begin{document}

\begin{frontmatter}



\title{Structure-Preserving Neural Ordinary Differential Equations for Stiff Systems}

\author[lanl]{Allen Alvarez Loya\fnref{DOEThanks,NSFThanks}} 
\ead{aalvarezloya@lanl.gov}

\author[lanl]{Daniel A.~Serino\fnref{DOEThanks}} 
\ead{dserino@lanl.gov}

\author[ut]{J.~W.~Burby\fnref{DOEThanks}} 
\ead{joshua.burby@austin.utexas.edu}

\author[gt]{Qi Tang\corref{cor1}\fnref{DOEThanks}}
\ead{qtang@gatech.edu}

\address[lanl]{Theoretical Division, Los Alamos National Laboratory, Los Alamos, NM}
\address[ut]{Department of Physics, University of Texas at Austin, Austin, TX}
\address[gt]{School of Computational Science and Engineering, Georgia Institute of Technology, Atlanta, GA}

\fntext[DOEThanks]{This work was partially supported by the U.S. Department of Energy Advanced Scientific Computing Research (ASCR) under DE-FOA-2493 “Data-intensive scientific machine learning”. It was also partially
supported by the ASCR program of Mathematical Multifaceted Integrated Capability Center (MMICC).}

\fntext[NSFThanks]{Alvarez Loya is supported by the U.S. National Science Foundation under NSF-DMS-2213261.}

\cortext[cor1]{Corresponding author: {\tt qtang@gatech.edu}.}

\begin{abstract}
Neural ordinary differential equations (NODEs)
are an effective approach for data-driven modeling of
dynamical systems arising from  simulations and experiments.
One of the major shortcomings of NODEs,
especially
when coupled with explicit integrators,
is its long-term stability,
which impedes their efficiency and robustness when encountering stiff problems. 
In this work we present a structure-preserving NODE approach that learns a transformation into a system with a linear and nonlinear split.
It is then integrated using an exponential integrator, 
which is an explicit integrator with stability properties comparable to implicit methods. 
We demonstrate that our model has advantages in both learning and deployment over standard explicit or even implicit NODE methods.
The long-time stability is further enhanced by the Hurwitz matrix decomposition that constrains the spectrum of the linear operator, therefore stabilizing the linearized dynamics.
When combined with a Lipschitz-controlled neural
network treatment for the nonlinear operator,
we show the nonlinear dynamics of the NODE are provably stable near a fixed point in the sense of Lyapunov.
For high-dimensional data, we further rely on an autoencoder performing dimensionality reduction
and Higham's algorithm for the matrix-free application of the matrix exponential on a vector.
We demonstrate the effectiveness of
the proposed NODE approach in various examples,
including the Robertson chemical reaction problem and the Kuramoto-Sivashinky equation.
\end{abstract}


\end{frontmatter}

\section{Introduction}
Data-driven reduced order models (ROMs) of dynamical
systems arising from high-resolution simulations or experiments 
can be used as efficient forward-model surrogates.
The learned surrogates can be deployed in  the traditionally expensive tasks of 
optimization and parameter inference. 
However, conventional data-driven ROMs for dynamical systems generally struggle with 
accurately approximating stiff dynamical systems \cite{kim2021stiff}
and maintaining long-term stability \cite{linot2023stabilized, serino2024intelligent}.
In this work, we address both of these challenges
in the context of neural ordinary
differential equations (NODEs) \cite{chen2018neural}, which are a class of
models that approximate dynamical systems continuously as
ODEs using neural networks.
The two ingredients that determine
the effectiveness of NODE-based models 
are the neural network architecture
and the numerical integrator. For the former we proposed a structure-preserving NODE, which integrates with a linear and nonlinear split and several small components that preserve key mathematical structures strongly.
For the latter we incorporate NODE with advanced time stepping through exponential integrators. 
The use of an exponential integrator improves the performance of explicit and semi-implicit integrators that 
are commonly used in packages such as diffrax \cite{kidger2021on} and torchode \cite{lienen2022torchode}.

In \cite{linot2023stabilized}, motivated by the inertial manifold theorem, we developed a special NODE approach that embeds a linear and nonlinear partition of the right-hand side (RHS) being approximated by two neural networks. This partitioning method provides better short-time tracking, prediction of the energy spectrum, and robustness to noisy initial conditions than standard NODEs. The recent work \cite{zhang2023semi,zhang2024semi} extended the linear and nonlinear partition to a semi-implicit NODE. In the forward pass of their training they solve a partitioned ODE using implicit-explicit Runge-Kutta (IMEX-RK) methods. 
The linear and nonlinear partition allows them to avoid solving a fully nonlinear system. The backward pass is handled through a discrete adjoint approach, resulting into a linear system to be solved. In general, the matrices from these linear systems are dense, requiring a direct solve which can be computationally expensive and have poor scalability. Our NODE uses an explicit approach based on exponential integrators, which do not require direct solvers but enjoy stability properties comparable to implicit approaches. 
To the best of our knowledge, exponential integration schemes have first 
been incorporated into NODEs in recent work~\cite{fronk2024training}
The authors show that exponential integrators can outperform implicit schemes for low-dimensional ODEs. 
However, this work shows the straightforward implementation of exponential integrators for high-dimensional systems to be too computationally expensive for practical purposes. 
Our work relies on several key components to overcome the challenge of training when applying such an exponential integrator with NODE to high-dimensional problems.

\begin{figure}[tb]
    \centering
\includegraphics[page=1,width=0.9\textwidth]{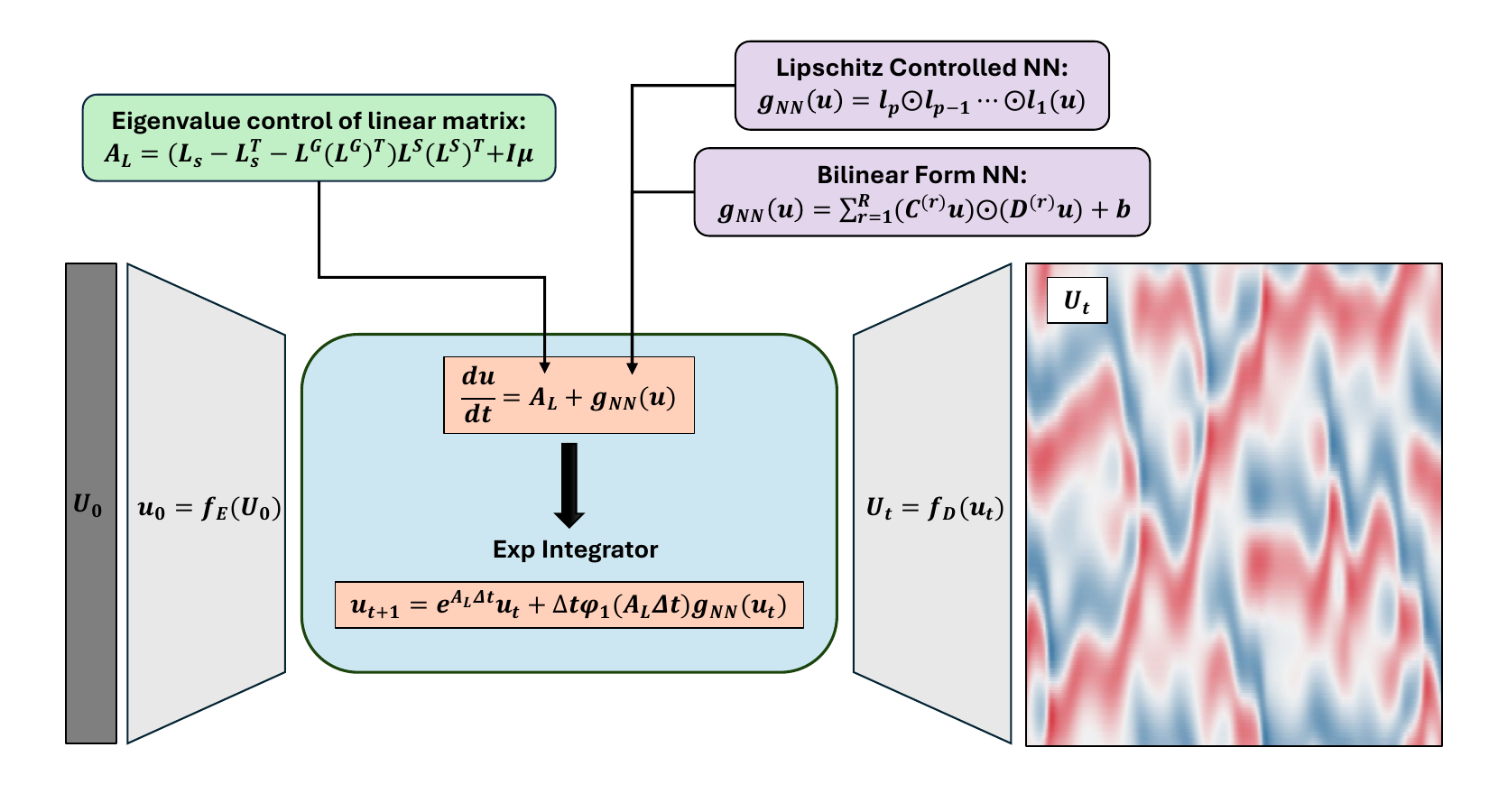}
    \caption{{Schematic of the autoencoder for 
    our structure preserving NODE. The dark grey block represents the initial data which is encoded using $f_E$ represented by the left light grey block. The latent space dynamics are evolved in the light blue block. The green and purple blocks represent the neural networks for the linear and nonlinear operators, respectively. These neural networks are used to push the dynamics forward using the exponential integrator defined in \eqref{eq:expSystemP1} represented by the light orange colored blocks. The data in the latent space is then decoded back into physical space using $f_D$. Decoded data is represented by the right light grey block. The solution, $U_t$ is represented by the right most block.
    }}
    \label{fig:autoEncoderSchematic}
\end{figure}

The contributions of our work include both structure-preserving 
enhancements and improvements in time stepping and training. 
In our approach, a coordinate transformation is learned via an autoencoder
in which the dynamics are described by a NODE with a
linear and nonlinear partitioning.
{Our approach is summarized by the schematic in
Figure~\ref{fig:autoEncoderSchematic}.}
{In related work, augmented Neural ODEs~\cite{dupont2019} expand the hidden state space to enhance expressivity and improve numerical stability.
Likewise, the structure-preserving formulation introduced here implicitly augments the state space through a linear-nonlinear partition, achieving similar advantages for modeling stiff systems.}

The linear part of the system is parameterized using a novel Hurwitz matrix decomposition based on~\cite{duan1998note}. This promotes long-term stability 
through bounding the real parts of the linear operator eigenvalues from above by a 
specified shift, $\mu\in\mathbb{R}$. 
The nonlinear operator is represented using either a bilinear form or a Lipshitz controlled neural network based on the innovations of our previous work~\cite{serino2024intelligent}. The bilinear form has the advantage of clearly separating the linear and nonlinear operators 
while a neural network
still contains a nonzero linear contribution in it's Taylor expansion.
In either case, we show that a local
Lipschitz bound on the nonlinear operator can provably guarantee stable
long-term predictions for certain dynamical systems.
{Lyapunov-based neural networks~\cite{lyapunov24} offer a related perspective, as they enforce stability through explicit Lyapunov-function parameterizations that guarantee monotonic energy decay along system trajectories. While sharing the goal of stability preservation, these methods are rooted in control-oriented formulations, whereas our approach achieves Lyapunov stability implicitly through Hurwitz spectral constraints on the linear operator and Lipschitz control on the nonlinear term, tailored specifically for stiff dynamical systems.}
The other major contribution of our work involves 
using an exponential integrator for the time integration.
We propose a matrix-free method for advancing the dynamics 
based on the algorithm developed in~\cite{al2011computing}. 
In our numerical examples, we show that the combination of structure-preserving
networks and exponential integrators enables us to robustly learn the Robertson 
system, which is a stiff system bridging time scales over multiple orders of 
magnitude.
For cases with high-dimensional data we use an autoencoder to significantly
reduce the dimensionality and aid in our training. 
We demonstrate using the Kuromoto-Sivashinsky equations that evolving dynamics on a low-dimensional latent space while preserving the linear and nonlinear partition allows a significant reduction in training time while retaining the stability of the model. 
We note that latent dynamics discovery has seen success in other scientific machine learning approaches (e.g.~\cite{fries2022lasdi, xie2024latent}).

Training and deploying neural ODEs in practical problems is often complicated by stiffness in the dynamics, which occurs when some or all of the solution components undergo rapid changes relative to transient behavior.
Stiff behavior forces explicit solvers to take prohibitively small timesteps. A number of recent
works have utilized ML to address the challenges of stiff systems.
In \cite{caldana2024neural} the stiffness issue is resolved by introducing a suitable reparameterization in time. The map produces a nonstiff surrogate that can be cheaply solved with an explicit method. In \cite{xie2024latent} the authors utilize a physics-assisted autoencoder to learn a low-dimensonal latent space of a collisional radiative model. In \cite{kim2021stiff} the authors utilize equation scaling to mitigate the stiffness problem. 
The method proposed in this work is capable of resolving the stiff dynamics by learning a suitable coordinate transformation into a new system composed of a linear operator capturing the stiffness of the system and a nonlinear operator which can be easily integrated using a standard first-order exponential integrator. This capability is highlighted in Roberton's chemical reaction problem in Section \ref{sec:Robertson}, in which conventional 
exponential integrators fail to eliminate 
the stiff time step restriction 
resulting from the original form of the equations.

The rest of the paper is organized as follows. Section \ref{sec:Parameterizations} presents the details behind the parameterizations of the linear and nonlinear parts of NODE along with implementation details. Section \ref{sec:Analysis} gives a result about the Lyapunov stability and provides an error bound for the NODE. 
Section \ref{sec:expIntegrators} provides a brief introduction to exponential integrators. 
The details related to training and their results are given in Section  \ref{sec:results}. Conclusive remarks are given in \ref{sec:conclusions}.

\section{Structure-Preserving NODE}\label{sec:Parameterizations}
This work focuses on 
dynamical systems that can be transformed into systems of the
form
\begin{equation}
\frac{du}{dt} = Au + g(u), 
\label{eq:linearPlusNonlinear}
\end{equation}
where $u \in \mathbb{R}^n$,  $g:\mathbb{R}^n\rightarrow\mathbb{R}^n$ is a nonlinear function,
and $A \in \mathbb{R}^{n \times n}$
has eigenvalues whose real parts are bounded above by a known estimate $\mu$.
Many semi-discretized time-dependent PDEs, such as the Navier-Stokes equations and Kuramoto-Sivashinsky equation, can be modeled using this framework.

Our proposed structure-preserving NODE takes the form
\begin{equation}
\frac{du}{dt} = A_{L}u + g_{NN}(u).
\label{eq:linearPlusNonlinearNODE}
\end{equation}
Here the linear trainable matrix $A_{L}$ is parameterized using  
$A_{L} = A_H + I\mu$, where $A_H$ has eigenvalues that
all have negative real part. The estimate for $\mu$ is problem specific. For example, if the linearized dynamics are known to be dissipative, then $\mu = 0$ is a suitable choice. 
In the case of Kuramoto-Sivashinky considered in
Section~\ref{sec:ksEQN},
we applied a Fourier analysis to the equations to determine an estimate.
Section~\ref{sec:HurwitzParam} describes
how $A_H$ can be parameterized using
a Hurwitz matrix decomposition.
The nonlinear operator, $g_{NN}$, is parameterized using either a bilinear form neural network or a Lipschitz controlled neural network,  
as described in Sections~\ref{sec:BilinearForm} and  \ref{sec:LipschitzControl}, respectively. 

\subsection{Hurwitz Matrix Parameterization}\label{sec:HurwitzParam}
A matrix where all the eigenvalues have negative real
parts is a Hurwitz matrix.
We develop a general decomposition for a Hurwitz matrix based on the following theorem.

\begin{theorem}
A matrix $A_H \in \mathbb{R}^{n \times n}$ is Hurwitz stable if and only if there exists a symmetric negative-definite matrix $G_s$,
skew symmetric matrix $G_u$, and a symmetric positive-definite matrix $S$ such that 
\begin{equation}
A_H = (G_s + G_u) S.
\label{eq:HurwitzMatrix}
\end{equation}
\end{theorem}
\noindent This theorem and its proof can be found in \cite{duan1998note}.

The Cholesky factorization, which involves a product of a lower triangular matrix with its transpose, is
a parameterization for 
symmetric positive definite matrices and involves
$\frac{n(n+1)}{2}$ unique parameters. The Cholesky factorization is used to parameterize $S$ 
and $G_s$,
\begin{subequations}
\begin{align}
S &= L^S(L^S)^{T},\\
G_s &= -L^G(L^G)^{T},
\end{align}
\end{subequations}
where $L^S$ and $L^G$ are lower triangular matrices.
The skew-symmetric matrix $G_u$ 
is formed using
\begin{align}
G_u = L_s - L_s^T,
\end{align}
where $L_s$ is a strictly lower triangular matrix, $L_s$
and therefore requires
$\frac{n(n-1)}{2}$ unique parameters.
The matrix $A_H$ is thus parameterized by a combination of $L^S$, $L^G$, and $L_s$.
The total number of parameters needed to describe $A_H$ is $\frac{3n^2+n}{2}$. 

\subsection{Bilinear Form Neural Network}\label{sec:BilinearForm}
One option to approximate 
the nonlinear part, $g$, is to 
use a bilinear form. 
By construction, the bilinear 
form has zero contribution
to the linear operator since it is built with quadratic forms.
This property is not guaranteed when using a vanilla feed-forward network for $g$. 
This creates a strict separation of the linear and nonlinear dynamics that allows recovery of the linear operator $A$ in an equation learning setting. 
We utilize the low rank-bilinear form network as defined in~\cite{serino2024intelligent}. It is defined by 
\begin{equation}
g_{NN}(u) = \sum \limits_{r = 1}^{R} (C^{(r)} u) \odot (D^{(r)} u) + b,
\label{eq:bilinearForm}
\end{equation}
where $u \in \mathbb{R}^{n}, \ C^{(r)} \in \mathbb{R}^{n \times n}, \ D^{(r)} \in \mathbb{R}^{n \times n},$
$b \in \mathbb{R}^{n \times 1}$, $\odot $ represents the Hadamard product (element-wise multiplication), and $R, \ 1 \leq R \leq n$ represents the rank of approximation which is chosen much smaller than $n$ during the training.

\subsection{Lipschitz Controlled Neural Network}\label{sec:LipschitzControl}
The second option we explore to approximate $g$ with is a neural network with a bounded Lipschitz constant.
As shown in Section~\ref{sec:LipunovStability},
having a Lipschitz bound on $g_{NN}$ has an impact on the long-term stability of the dynamics. We introduce a novel Lipschitz-controlled neural network architecture.
{
\definition\label{def:lcnn}
The network is a feedforward network given by 
\begin{align}
g_{NN}(u) =  l_p \odot l_{p-1} \dots  \odot l_1(u)
\end{align}
where
\begin{align}
  l_i(z) = \sigma(B_i z + b_i), \qquad
    B_i = L^{1/p} \frac{W_i}{\max(1,\, \sqrt{\|W_i\|_{\infty} \|W_i\|_1})},
\end{align}
$W_i\in\mathbb{R}^{n_i\times n_{i+1}}$ are trainable matrices, $n_1=n_{p+1}=n$, 
$n_2=\dots=n_p=d$ is a chosen latent dimension and $\sigma$ is a $1$-Lipschitz activation function.
$L>0$ is the desired Lipschitz constant of the network $g_{NN}$. 
}

Note that the $1$-norm and $\infty$-norm of a matrix are the maximum absolute column sum and the maximum absolute row sum, respectively. 
Therefore, the realization of $B_i$ can be implemented in a manner compatible with automatic differentiation.
The following theorem shows $g_{NN}$ has the desired property.
{\theorem
The network $g_{NN}(u)$ given in Definition \ref{def:lcnn} is $L$-Lipschitz.
}

\begin{proof} 
First, we show each layer $l_i$ is $L^{1/p}$-Lipschitz. \das{Let $z_1$ and $z_2$ be inputs for the layer $l_i$.
Since $\sigma$ is a 1-Lipschitz 
activation function, }
\begin{align*}
\| l_i(z_1) - l_i(z_2) \|_{2} &= \| \sigma (B_i z_1 + b_i) - \sigma (B_i z_2 + b_i) \|_2\\
& \leq  \| B_i z_1 + b_i - (B_i z_2 + b_i) \|_2\\
& \leq \| B_i (z_1 - z_2) \|_2\\
& \leq  \| B_i \|_2\|z_1 - z_2 \|_2.
\end{align*}
The operator norm of $B_i$ may be bounded as follows
\begin{align*}
\|B_i \|_2 &=L^{1/p} \frac{\|W_i \|_2}{\max\left(1,\sqrt{\|W_i\|_1\|W_i\|_{\infty}}\right)}\leq L^{1/p},
\end{align*}
in which the inequality follows a special case of H\"{o}lder's inequality, $\|W_i\|_2 \leq \sqrt{\|W_i\|_1 \|W_i\|_{\infty}}$. Therefore, each layer is $L^{1/p}$-Lipschitz. It follows that the network $g_{NN}$ is $L$-Lipshitz.
\end{proof}

Our previous work \cite{serino2024intelligent} introduced a bi-Lipschitz affine transformation (bLAT) network,
achieving Lipschitz control using a parameterization 
based on the singular value decomposition (SVD). 
The parameteization presented here is more efficient since the SVD is more expensive to evaluate than the $1$ and $\infty$ norms. 
However, unlike the above, the bLAT network is bi-Lipschitz. 

\subsection{Autoencoder}
Inspired by the seminal work \cite{lusch2018deep} for complex dynamical systems involving high-dimensional data, 
an autoencoder produces a 
low dimensional representation of the dynamics which can be integrated efficiently.
For high-dimensional problems like the Kuramoto-Sivashinsky equation, we find it is necessary to include the dimension reduction in structure-preserving NODE to avoid training directly on a high-dimensional space, which is more challenging and expensive. 
For low-dimensional problems such as the Robertson problem in Section \ref{sec:Robertson}, the autoencoder is additionally used to learn a transformation into coordinates that yields a system of the same dimension but with stiffness coming from the linear part. Instead of simplifying our model by reducing the dimension, the system is transformed such that our exponential integrator is able to integrate utilizing large timesteps. 


We consider some dynamical system
for $U(t)\in \mathbb{R}^{m}$ and 
assume that it can be transformed into a variable $u\in \mathbb{R}^n$, $n < m$ 
where the dynamics are described by a system of the form 
\eqref{eq:linearPlusNonlinear}.
Two feed-forward neural networks are trained $f_E(U;\alpha)$ and $f_D(u;\beta)$, which represent the encoder and decoder and have network parameters $\alpha$ and $\beta$, respectively. The network $f_E:\mathbb{R}^m \to \mathbb{R}^n$ takes in the initial condition in physical space and network parameters and outputs the initial condition in the latent space $\mathbb{R}^n$ with $n\le m$. Once the initial condition is projected onto the latent space we use a time integrator to evolve the dynamics forward in time.
By evolving the dynamics in the low-dimensional latent space $\mathbb{R}^n$, efficiency 
is gained in the numerical integration step.
In mathematical symbols we have $f_E(U_0) = u_0$ then we evolve $u_0$ several timesteps (using the exponential integrator in this work) to obtain $\lbrace u_1, u_2, \dots, u_t \rbrace $. After this, these are each transformed back into the full space data using the decoder $U_i = f_D(u_i)$, which are then used to compute the loss term. A schematic of this process is given in Figure \ref{fig:autoEncoderSchematic}.

\section{Analysis of the Network}\label{sec:Analysis}
In this section two results regarding the proposed structure-preserving NODE are presented. 
The results include a Lyapunov stability theorem in Section~\ref{sec:LipunovStability}
and an error bound for a trained model in Section~\ref{sec:errorBound}.

\subsection{Lyapunov Stability of the Network}\label{sec:LipunovStability}
The NODE is parameterized such that the spectrum of the 
linear operator $A_{L}$ is bounded from above by some $\mu\in\mathbb{R}$.
In the case of $\mu < 0$, we show that  
any fixed point is an asymptotically stable attractor 
for the dynamical system defined by the NODE when a weak constraint on the nonlinear part is satisfied. 
The following theorem is applicable to the case of learning near any fixed point $u_e$.


{\theorem
If the structure-preserving NODE \eqref{eq:linearPlusNonlinearNODE} is parameterized such that 
the linear operator $A_{L}$
satisfies ${\rm Re}(\lambda(A_L)) \le \mu < 0$ and 
$g_{NN}$ is an $L$-Lipschitz function
in the neighborhood of a fixed point $u=u_e$,
then for sufficiently
small $L$ the NODE is stable in the sense of Lyapunov at the fixed point $u=u_e$.}

\begin{proof}
First we center around the fixed point 
using the change of coordinates,
$u = \tilde{u} + u_e$. This results in a 
system of the form
\begin{align}
\frac{d}{dt}\tilde{u} = A \tilde{u} + \tilde{g}_{NN}(\tilde{u})
\end{align}
where $\|\tilde{g}_{NN}\| \le L \|\tilde{u}\|$
near $\tilde{u}=0$.
Therefore, without loss of generality, 
we prove the theorem using the notation of the 
original system assuming $u_e=0$
and $\|g\| \le L \|u\|$.
It suffices to show the existence of a Lyapunov function, $V(u)$, such that $V(u) \geq 0,
V(u) = 0$ if and only if $u = 0$, and $\frac{d}{dt} V(u) \leq 0$ for all $t>0$.
Let $A_{L} = T J T^{-1}$ be the Jordan decomposition for $A_{L}$.
According to \cite{harier1993}, for any $\epsilon >0$ it is possible to write 
\begin{equation}
J = T^{-1}A_{L}T = \rm{diag} \left \lbrace \begin{pmatrix} \lambda_1 & \epsilon & \\
&\ddots & \epsilon\\ & & \lambda_1\end{pmatrix},\begin{pmatrix} \lambda_2 & \epsilon & \\
&\ddots & \epsilon \\ & & \lambda_2\end{pmatrix},\dots \right \rbrace,
\end{equation}
where $\lambda_1, \lambda_2, \dots$ are 
the unique eigenvalues of $A_L$.
Let $z = T^{-1} u \in \mathbb{C}^n$.
We define $V(u) = \langle z, z \rangle = z^* z$
where $\langle \cdot, \cdot \rangle$
is the complex inner product, then
\begin{align*}
\frac{d}{dt} V(u) & = \frac{d}{dt} \langle z, z \rangle\\
& = \left(\frac{d}{dt} z\right)^* z + z^* \frac{d}{dt} z \\
&= z^* (J + J^*) z + 2{\rm Re}\langle z, T^{-1} g_{NN}(Tz) \rangle
\end{align*}
It follows that $\frac{1}{2}(J + J^{*})$ is block-diagonal where each block is a tridiagonal matrix with $\rm{Re} \lambda_i$ on the diagonal and $\frac{\epsilon}{2}$ on the superdiagonal and subdiagonal. The eigenvalues are given by $\rm{Re}\lambda_i + \epsilon \cos \left(\frac{\pi k}{m+1}\right)$ with $k = 1, \dots, m$. 
Recall that $\rm{Re}\lambda_i\le \mu$.
Therefore, $\frac{1}{2} z^* (J + J^*) z \le \mu + \epsilon$
and we have the bound
\begin{align*}
\frac{d}{dt} V(u) &\leq 2(\mu + \epsilon) \|z\|^2 + 2\|z\|\|T^{-1}\|\|g_{NN}(Tz)\|
\end{align*}
\bogus{
\begin{align*}
\frac{d}{dt} V(u) & = \frac{d}{dt} \|z\|^2\\
& = 2 {\rm Re}\left\langle z, \frac{d}{dt}z \right\rangle\\
& = 2 {\rm Re}\langle z, J z \rangle + 2{\rm Re}\langle z, T^{-1} g_{NN}(Tz) \rangle 
&\leq 2\kappa(J) \|z\|^2 + 2\|z\|\|T^{-1}\|\|g_{NN}(Tz)\|
\end{align*}}
Due to the Lipschitz bound, $\|g_{NN}(u)\|\le L\|u\|$, we obtain the following bound.
\begin{align*}
\frac{d}{dt} V(u) 
&\leq 2(\mu + \epsilon + L \|T^{-1}\| \cdot \|T\|)V(u)
\end{align*}
For any $\epsilon$ satisfying $\epsilon + \mu < 0$
we can choose
$L < \frac{|\mu + \epsilon|}{\|T^{-1}\| \cdot \|T\|}$,
to make the right-hand side is negative. 
Therefore, for
$L$ sufficiently small, $V(u)$ is a Lyapunov function. This implies Lyapunov stability.
\end{proof}

It is important to note that the above Lyapunov stability result of the proposed network is only applicable to a neighborhood of an equilibrium. It does not extend to other attractors or chaotic systems. Nevertheless, we show later empirically that this predicts a good long-time solution for chaotic dynamics when evaluated using appropriate statistical measures. The above assumption of addressing network stability near equilibrium is motivated by the practical stiff dynamics such as chemical reaction and collisional radiative models \cite{xie2024latent}, where the dynamics of interest is a transition from various initial conditions to the corresponding equilibria.

\subsection{Bounding the Error of a Learned Model}\label{sec:errorBound}
Given an estimate on the approximation error of the learned operators,
a bound can be obtained for the error between the 
ground truth and learned dynamics.

\begin{theorem}
Suppose that $u(t)$ is a solution of the system of differential equations $u' = Au+g(u)$, $u(t_0) = u_0$,  $v(t)$ is and approximate solution generated by structure-preserving NODE where $\| v\| \leq \mathcal{V}$. If $u$ and $v$ are differentiable, then the error between the two trajectories is bounded by  
\begin{equation}
\|u(t) - v(t) \| \leq \rho e^{(||A|| + L)(t-t_0)} + \frac{(\| dA\| + \|dg\|) \mathcal{V}}{||A|| + L} \left(e^{(||A|| + L)(t-t_0)}-1 \right).
\label{eq:errorBound}
\end{equation}
Here $\rho$ is the error in the initial condition. $dA$ and $dg$ is the difference between the neural network operators and the ground truth for $A$ and $g$, respectively, and $L$ is the Lipschitz constant for $g_{NN}$. 
\end{theorem}

\begin{proof}
If $m(t) = \|v(t) - y(t)\|$, then 
\begin{align*}
\frac{dm(t)}{dt} &= \lim \limits_{h \to 0} \frac{m(t+h) - m(t)}{h} \\
&= \lim \limits_{h \to 0} \frac{\|v(t+h) - u(t+h)\| - \|v(t)-u(t)\|}{h} \\
&\leq \lim \limits_{h \to 0} \frac{\|v(t+h) - u(t+h) - (v(t)-u(t))\|}{h} \\
&=\|v'(t) - u'(t)\|\\
&= \| A_{L}v + g_{NN}v - (Au+gu) \|\\
&= \| A_{L}v + g_{NN}v - (Av+gv) + (Av+gv) - (Au+gu) \|\\
&\leq (\|dA\|+\|dg\|)\|v\| + (\|A\| + L)\|v-u\|\\
&\leq (\|dA\|+\|dg\|)\mathcal{V} + (\|A\| + L)m(t)
\end{align*}
The result of \eqref{eq:errorBound} then follows directly from the classical differential form of Gronwall inequality.
\end{proof}

\section{Exponential Integrators}\label{sec:expIntegrators}
The proposed NODE form $\eqref{eq:linearPlusNonlinear}$ assumes a linear and nonlinear split where the stiffness of the system shall come primarily from the linear matrix $A$.
For such a system, the exponential integrator is an ideal technique to overcome its stiffness, and therefore it is leveraged in this work in the NODE training.
In this section a brief review exponential integrators is presented. For an in depth introduction see  \cite{hochbruck2010exponential}. Consider a system that may be written in the form 
\begin{equation}
\frac{du}{dt} = Au + g(u), 
\quad u(t_0) = u_0, \quad t \geq t_0,
\label{eq:linearPlusNonlinear2}
\end{equation}
where $u \in \mathbb{R}^n, \ A \in \mathbb{R}^{n \times n}$, and $g:\mathbb{R}^n\rightarrow\mathbb{R}^n$ is a nonlinear function. 
The solution of \eqref{eq:linearPlusNonlinear2} satisfies the nonlinear integral equation 
\begin{equation}
u(t) = e^{(t-t_0)A}u_0 + \int \limits_{t_0}^t e^{(t - \tau)A} g(\tau,u(\tau)) \ d\tau.
\end{equation}
A $p$-th order Taylor series expansion of $g$ about $t_0$ gives 
\begin{equation}
u(t) = e^{(t - t_0)A}u_0 + \sum \limits_{k = 1}^{p} \varphi_k((t - t_0)A)(t - t_0)^k g_k
+ O((t-t_0)^{p+1}),
\label{eq:expIntPthOrder}
\end{equation}
where 
\[g_k = \left. \frac{d^{k-1}}{dt^{k-1}}g(t,u(t)) \right|_{t = t_0}, \quad \varphi_k(z) = \frac{1}{(k-1)!} \int \limits_{0}^{1} e^{(1-\theta)z}\theta^{k-1} \ d\theta, \quad k \geq 1.\]
\bogus{
A $p$-th order approximation of \eqref{eq:expInt} is given by
\begin{equation}
u = e^{(t - t_0)A}u_0 + \sum \limits_{k = 1}^{p} \varphi_k((t - t_0)A)(t - t_0)^k g_k,
\label{eq:expIntPthOrder}
\end{equation}
}
The functions $\varphi_l(z)$ satisfy the recurrence relation 
\[ \varphi_{l+1}(z) = \frac{\varphi_{l}(z) - \frac{1}{l!}}{z}, \quad \varphi_0(z) = e^z,\]
and have the Taylor series expansion 
\[\varphi_l(z) = \sum \limits_{k = 0}^{\infty} \frac{z^k}{(k+l)!}.\]

The integrator used in this study is exponential time differencing (ETD) of order 1 \cite{minchev2005review}.  The ETD methods are based on approximating the nonlinear term by an algebraic polynomial. Note that there are many more exotic versions of exponential integrators such as exponential Rosenbrock methods and Lawson methods \cite{hochbruck2010exponential}.  

Evaluating the $\varphi$ functions is a nontrivial task. The computation of $\varphi_1$ is a well known problem in numerical analysis \cite{higham2002accuracy,hochbruck1998exponential,lu2003computing}.  Ref.~\cite{al2011computing} shows that it is possible to avoid computing the $\varphi$ functions \eqref{eq:expIntPthOrder} by solving a system with a single matrix exponential of size $(n + p) \times (n + p)$, where $n$ is the dimension of the system and $p$ is the order of the approximation. The system \eqref{eq:linearPlusNonlinear} is approximated using 
\begin{equation}
\hat{u}(t) = \begin{bmatrix} I_n & 0 \end{bmatrix} \exp \left((t-t_0) \begin{bmatrix} A & W\\ 0 & J \end{bmatrix}  \right) \begin{bmatrix} u_0 \\ e_p \end{bmatrix},
\label{eq:expSystem}
\end{equation}
where 
$e_p$ is a vector with a $1$ in the $p^{\rm th}$ element
and $0$'s elsewhere 
and
$W\in\mathbb{R}^{n\times p}$ and $J\in\mathbb{R}^{p\times p}$
are defined to be
\begin{align}
J = \begin{bmatrix} 0 & I_{p-1} \\ 0 & 0 \end{bmatrix},\qquad
    W_{ij} = 
    \begin{cases}
        g_k, & j=p-k+1 \\
        0, & \text{otherwise}
    \end{cases}.
\end{align}
$I_n$ and
$I_{p-1}$ are identity matrices with sizes
$\mathbb{R}^{n\times n}$ and
$\mathbb{R}^{p-1\times p-1}$, respectively.
When $p = 1$, then  $W = u_0$ and $J = 0$. Thus, for $p = 1$, \eqref{eq:expSystem} reduces to
\begin{equation}
\hat{u}(t) = \begin{bmatrix} I_n & 0 \end{bmatrix} \exp \left((t-t_0) \begin{bmatrix} A & g_0\\ 0 & 0 \end{bmatrix}  \right) \begin{bmatrix} u_0 \\ 1 \end{bmatrix}.
\label{eq:expSystemP1}
\end{equation}

A matrix-free algorithm developed in \cite{al2011computing} is used to compute \eqref{eq:expSystemP1}. This algorithm adapts the scaling and squaring method \cite{higham2005scaling} by computing $e^{t A}b \approx(T_m(s^{-1} t A))^s b$,
where $T_m$ is a truncated Taylor series.
Algorithmically, the computation is performed using
\begin{align}
e^{t A} b \approx \texttt{expmv}(A, b, t) = f_{s},
\end{align}
where $f_s$ is defined by the following recurrence relations
\begin{subequations}
\begin{alignat}{2}
b_{1, 0} &= b, \\
b_{i,k} &= \frac{t}{s k} A b_{i, k-1}, \qquad &&i=1,\dots,s, \quad k=1,\dots,m \\
b_{i+1,0} = f_i &= \sum_{k=0}^m b_{i,k}, \qquad &&i=1,\dots,s  \label{eq:recursesum}
\end{alignat}
\end{subequations}
The choice of $s$ and $m$ impact both
accuracy and efficiency.
The number of matrix vector multiplications is $sm$ and
an analysis for the optimal choice of $s$ and $m$ may be found in \cite{al2011computing}.
In practice, we choose $s$ to be proportional to 
the maximum time step of the training data and
the norm of $A$,
$s \sim \Delta t_{\rm max} \|A\|_2$,
and we choose $m$ such that~\eqref{eq:recursesum}
converges to some tolerance, $\tau$.
Utilizing the matrix-free computation eliminates the need to compute the matrix exponential in equation \eqref{eq:expSystem}, 
and can provide a faster alternative when the dimension
of $A$ is large.
Moreover, the computation of a matrix exponential is required for each trajectory due to the time dependency. Batching of data is improved since a matrix-free algorithm allows for vectorization. 
This form is used to evolve the dynamics for the structure-preserving NODE in Section~\ref{sec:results}. 
Note that higher-order versions 
of this scheme require the derivatives of the solution.

\das{
The matrix-free exponential integrator used in this work is currently implemented as a prototype in high-level Python, and its performance therefore reflects an unoptimized version of the algorithm. Its computational overhead depends on factors such as the number of matrix-vector multiplications, the system dimension, and the convergence properties of the chosen polynomial or rational approximation, which in turn depend on the spectrum of the matrix. In practice, the matrix-free approach scales favorably for large or sparse systems where direct matrix exponentiation or repeated linear solves are computationally expensive. Future compiled or GPU-accelerated implementations have the potential to substantially reduce the overhead, and we expect that further development of this algorithm by the numerical analysis community could yield significant performance gains.
}

Exponential integrators are explicit methods, but enjoy stability properties comparable to implicit methods. These methods are used to solve two types of stiff problems. The first are systems with a Jacobian whose eigenvalues has large negative real parts. These systems usually arise in the discretization of parabolic partial differential equations. The second type of problems are systems that are highly oscillatory in nature these systems have purely imaginary eigenvalues. The focus of this paper is on the former.

\section{Numerical Examples}
\label{sec:results}

In this section the structure-preserving NODE framework is applied to three examples, including a weakly nonlinear system with transient growth, 
the Robertson's chemical reaction problem, and the Kuramoto-Sivashinsky equation.
The first example highlights the 
benefits of using an exponential integrator compared to 
an implicit-explicit (IMEX) integration scheme, 
which includes improved accuracy when learning
dynamics using large time steps.
Moreover, we demonstrate that the properties of the exponential integrator allow us to deploy our NODE using different time steps than those used for training.
The Robertson example
highlights the ability to deal with a problem which has stiff dynamics with time scale separations. 
Finally, the Kuramoto-Sivashinsky equation
demonstrates learning chaotic dynamics using
our NODE.

\paragraph{Training details}
The input to the NODE is a set of initial conditions and the output a the set of trajectories evolved from the initial conditions. For the linear part of the proposed NODE, the proposed parameterization is used for the Hurwitz matrix learning. For the nonlinear part, either a bilinear form or a feed-forward neural network with Lipschitz controlled layers is used. The networks have 2 layers with 100 to 200 hidden dimensions. The activation function is $\tanh$. The autoencoder is made up of two simple feed-forward neural networks, one for the encoder and one for the decoder. The loss is computed using an approximation to the time integral of the mean squared error of the solution, given by 
\begin{align}
L(t, U, \tilde{U}) = 
    \sum_{n=0}^{N_t-1} \frac{1}{2}\left(
    (U_{n+1} - f_D(\tilde{u}_{n+1}))^2
    + (U_{n} - f_D(\tilde{u}_{n}))^2
    \right)(t_{n+1} - t_{n}),
    \label{eq:LossPrediction2}
\end{align}
where the initial latent variable is projected through $\tilde{u}_0=f_E(U_0)$. Note that the loss is computed in the original coordinates; thus, the autoencoder is optimized along with the NODE during optimization step.
The inclusion of this autoencoder network is optional in our framework.
Our implementation of the NODE uses Jax~\cite{jax} as a backend. 
Automatic differentiation is used to compute gradients of the 
loss function with respect to the trainable parameters,
and the Adam optimizer from the Optax~\cite{optax} package is 
used for optimizations.

\subsection{Taming Transient Growth}
In this first example we apply our structure-preserving NODE to learn 
transient growth for a simple system. 
We focus on a weakly nonlinear dynamical system in which the right-hand-side can be decomposed into the sum of linear and nonlinear operators,
\begin{equation}
\frac{dx}{dt} = Ax + \epsilon \sin(x), \quad x(0) = x_0,
\label{eq:simpleLinear}
\end{equation}
where $\sin(x)$ is interpreted to be a component-wise application of the function.
The linear operator, $A$, is chosen such that
in the limit of $\epsilon\to 0$
there is initial transient growth followed by long-term decay. This example is influenced by work done in~\cite{krovi2023improved} where the author analyzes how the norm of the matrix exponential contributes to the runtime of quantum algorithms for linear ordinary differential equations
(i.e., $\epsilon=0$). 
Some relevant quantities used in their analysis are 
\begin{subequations}
\begin{align}
\alpha(A) &= \max \{\text{Re}(\lambda) \;\vert\; \lambda \in \sigma(A) \} \label{eq:alpha} \\
\gamma(A) &= \max \left\{\lambda \;\vert\; \lambda \in \sigma\left((A + A^{T})/2\right) \right\}, \label{eq:mu}
\end{align} 
\end{subequations}
where $\sigma(\cdot)$ represents the the set of
eigenvalues of the argument.
The values of {$(\alpha, \gamma)$} describe the long term behavior of the norm of $e^{At}$. We focus on the case when $\alpha < 0$, which implies $\lim \limits_{t \to \infty}||e^{At}|| \to 0$. While this gives the long term behavior, the quantity $\alpha$ does not give any insight into the short term behavior around $t = 0$. The sign of $\gamma$ tells us whether $||e^{At}||$ has an initial grown or decay around $t = 0$, $\gamma > 0$ corresponding to growth and $\gamma < 0$ corresponding to decay. To illustrate how $\gamma$ impacts the short timescale we consider two matrices in $\mathbb{R}^{2\times 2}$,
\begin{equation}
A_{\text{decay}} = \begin{bmatrix} -2 & 1 \\ 0 & -2 \end{bmatrix}, \quad A_{\text{bump}} = \begin{bmatrix} -2 & 10 \\ 0 & -2 \end{bmatrix}.
\label{eq:decayAndBumpMatrices}
\end{equation}
As both matrices have eigenvalues equal to $-2$,
both systems have $\alpha < 0$, but 
$\gamma(A_{\text{decay}})  < 0 <\gamma(A_{\text{bump}})$. 
This shows that $e^{tA_{\text{decay}}}$ has norm that decays for all $t > 0$ and $e^{tA_{\text{bump}}}$ has an initial growth in the norm before the decay as is shown in Figure \ref{fig:expNorm}.
\begin{figure}[htb]
    \centering
    \includegraphics[width=.65\textwidth]{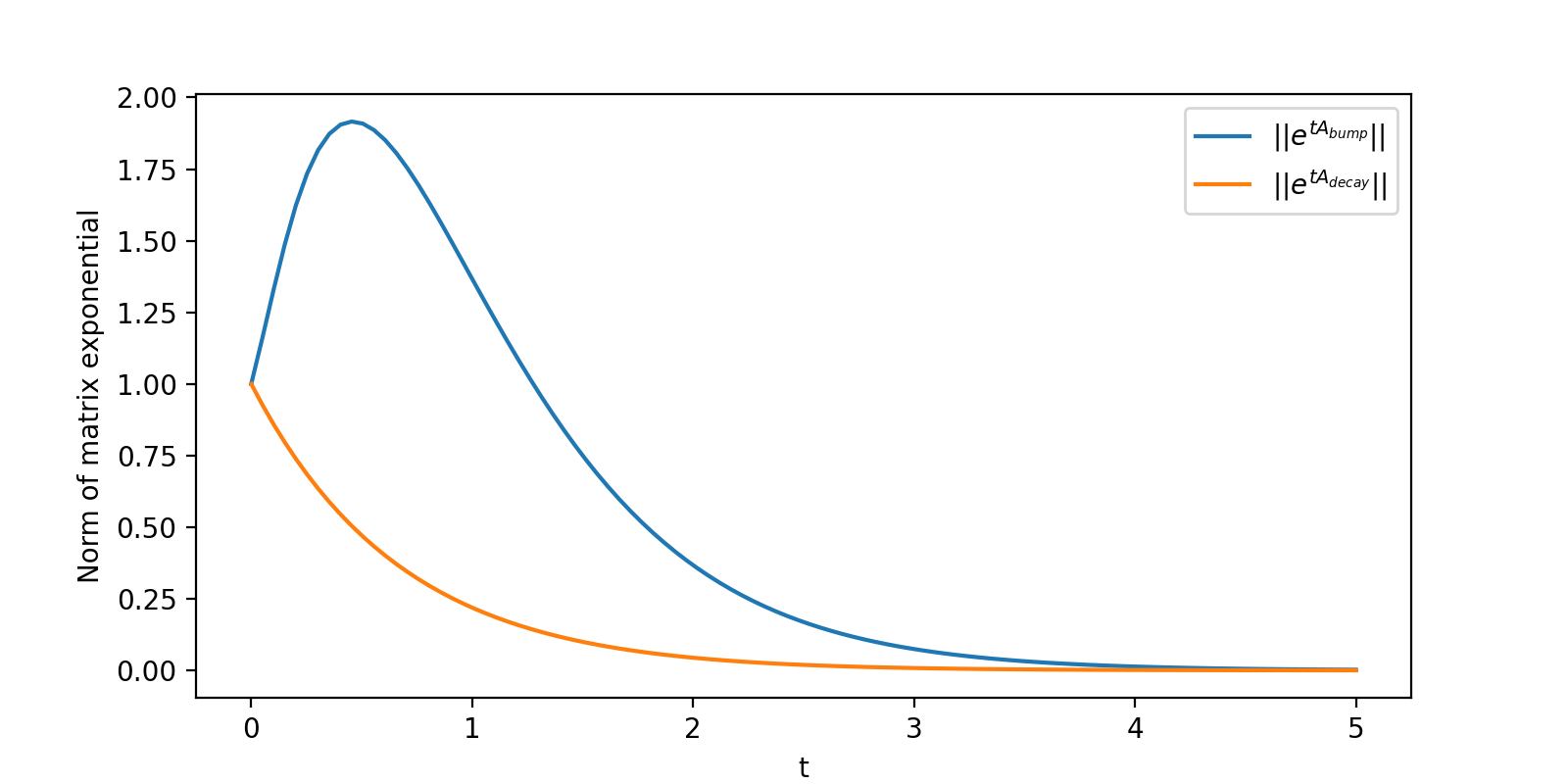}
    \caption{Plot of $||e^{tA_{\text{bump}}}||$ and $||e^{tA_{\text{decay}}}||$ versus time $t$. The blue curve corresponds to positive $\gamma$ defined in equation \eqref{eq:mu} and the orange corresponds to negative $\gamma$.
    }
    \label{fig:expNorm}
\end{figure}

For our experiments, we consider $A=A_{\text{bump}}$.
and $\epsilon\in\{0, 1.0\}$.
The initial bump in the norm $||e^{tA_{\text{bump}}}||$ makes the  dynamics more challenging to learn, as this implies initial transient growth before the long term decay. 
Data is generated using an adaptive Runge-Kutta-Fehlberg scheme. There are 1000 initial conditions evolved using 50 adaptive timesteps 
and a relative tolerance of $10^{-8}$. 
For training, only the data at a few points in time are utilized.
For $\epsilon = 0.0$ only the initial and final timesteps are utilized and for $\epsilon = 1.0$ only timesteps $0, 10, 20, 30, 40$, and $50$ are utilized. Note that the exponential integrator can learn the dynamics for $\epsilon = 0.0$ with a single timestep, but the IMEX scheme is unable to capture the dynamics. In our network, the nonlinear term is treated using
a bilinear form network when $\epsilon =0.0$ and Lipschitz controlled layers when $\epsilon = 1.0$. In both cases 2 layers and 100 hidden dimensions are used.
The Adam optimizer from Optax is used for training with a learning rate of 0.01
and batch size of 1000. 
The training is run for 1000 iterations for each case.
For comparison, the integration scheme is replaced with the implicit-explicit (IMEX) SSP2(2,2,2) L-Stable Scheme defined in Table II of~\cite{pareschi}. The predictions of the models for $\epsilon=0.0$ are shown in Figure \ref{fig:AbumpLinearModels}
and for $\epsilon=1.0$ are shown in Figure~\ref{fig:AbumpNonLinearModels}.
Learning with the exponential integrator yields more accurate results than the IMEX integrator for both values of $\epsilon$, the difference is more pronounced for the $\epsilon = 1.0$. 

\begin{figure}[h]
    \centering
    \includegraphics[width=0.45\textwidth]{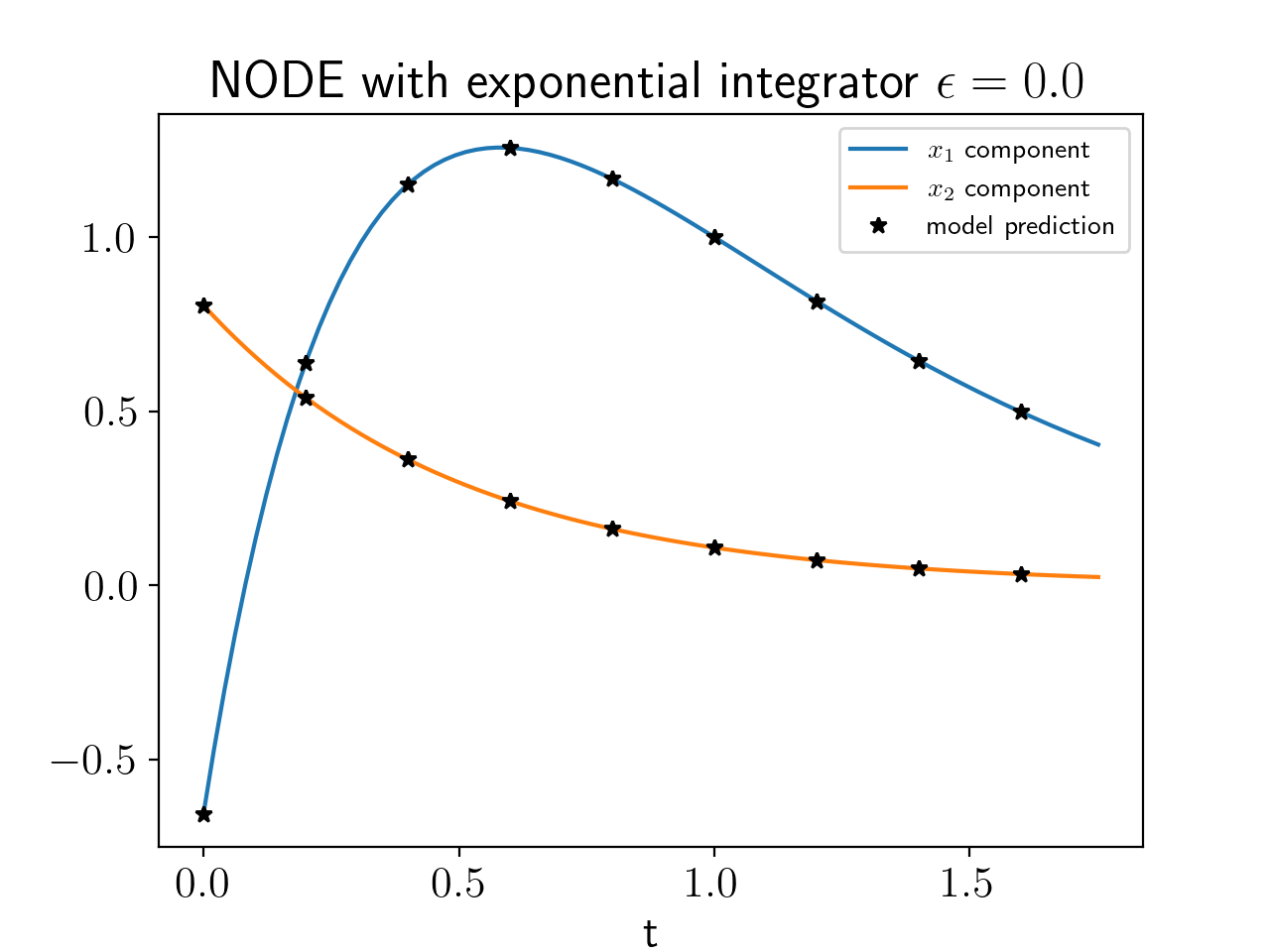}
    \includegraphics[width=0.45\textwidth]{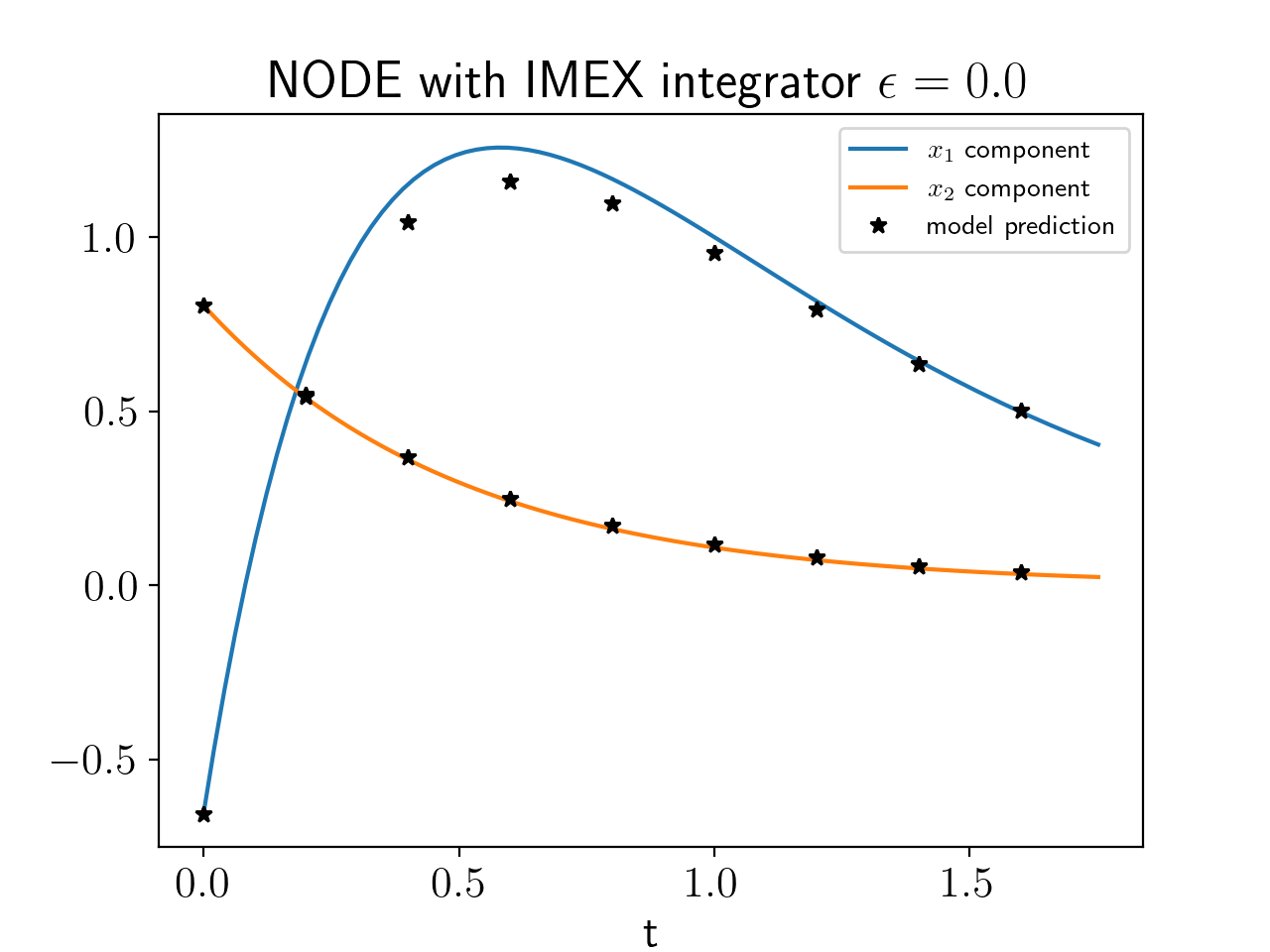}
    \caption{Prediction of the weakly nonlinear system using and exponential integrator (Left) and and IMEX integrator (Right) using a bilinear form for the nonlinear term. Black dots represent the model prediction using timesteps of $\Delta t = 0.2$. The blue and orange curves  represents the true dynamics for $x_1$ and $x_2$, respectively.
    }
    \label{fig:AbumpLinearModels}
\end{figure}

\begin{figure}[h]
    \centering
    \includegraphics[width=0.45\textwidth]{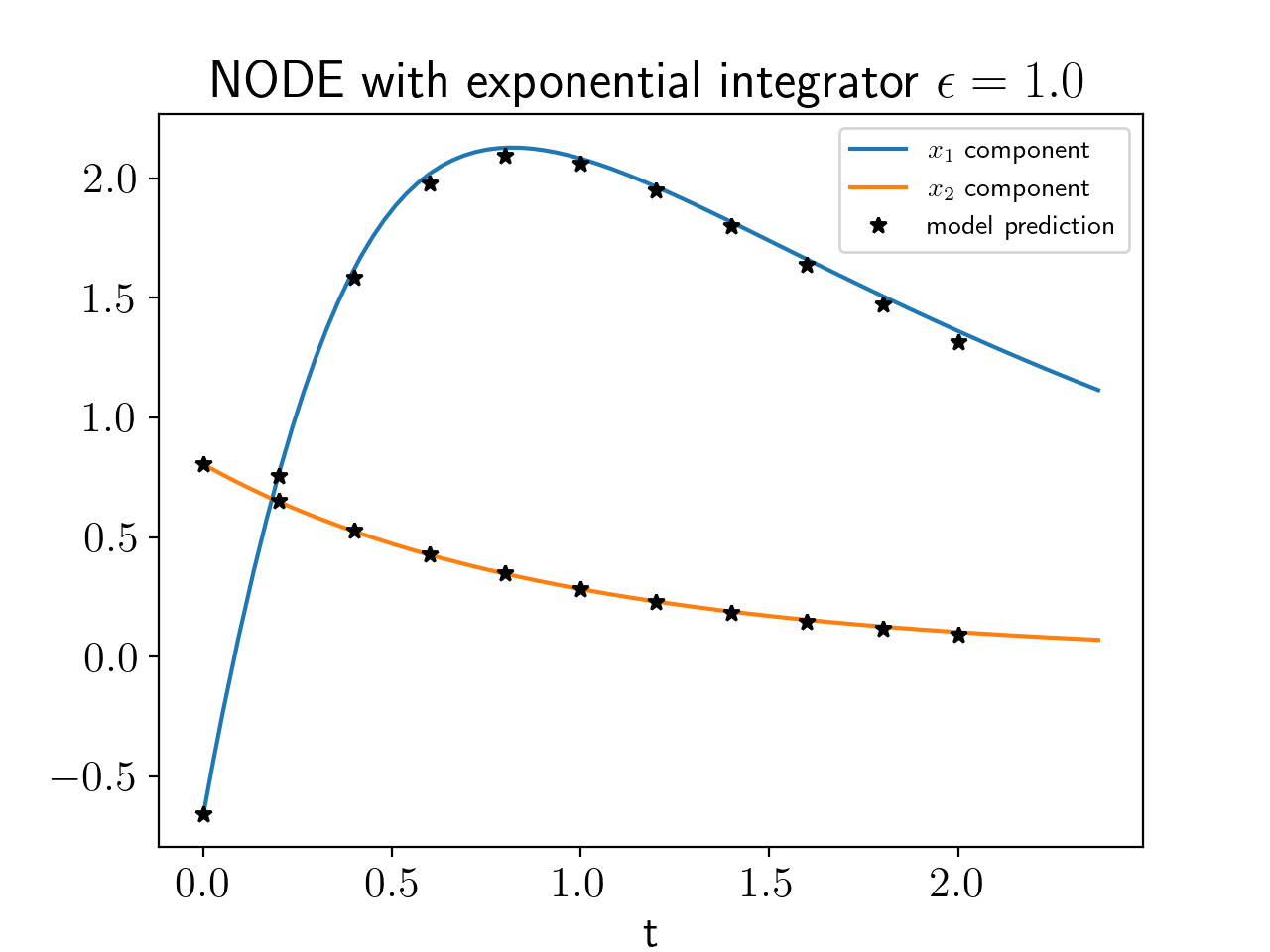}
    \includegraphics[width=0.45\textwidth]{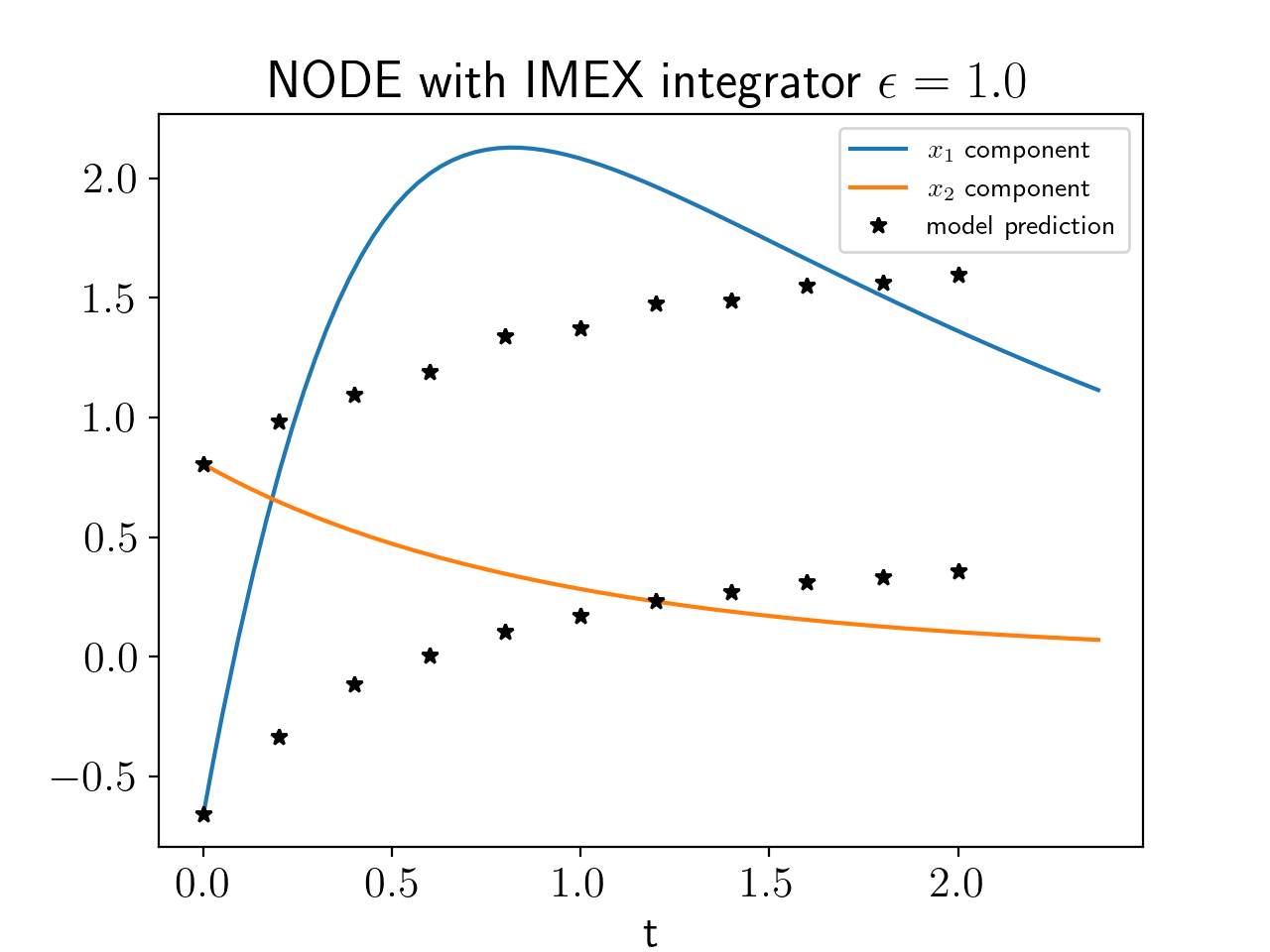}
    \caption{Prediction of the weakly nonlinear system with $\alpha = 1.0$ using an exponential integrator (Left) and an IMEX integrator (Right) using a Lipschitz controlled neural network for the nonlinear term. Black dots represent the model prediction using timesteps of $\Delta t = 0.2$. The blue and orange curves  represents the true dynamics for $x_1$ and $x_2$, respectively.
    }
    \label{fig:AbumpNonLinearModels}
\end{figure}

\subsection{Robertson's Chemical Reaction}\label{sec:Robertson}


The Robertson problem is a stiff three-dimensional system given by 
\begin{align*}
\frac{dy_1}{dt} &= -0.04y_1 + 10^4 y_2y_3,\\
\frac{dy_2}{dt} &=  0.04y_1 -10^4y_2y_3 - (3\times 10^7)y_2^2,\\
\frac{dy_3}{dt} &= (3 \times10^7)y_2^2.
\end{align*}
This system describes the kinetics of an autocatalytic chemical reaction and is a benchmark problem for stiff ODE solvers. 

The training data consists of 1000 trajectories generated using MATLAB's {\tt ODE15s} solver. Each trajectory has initial condition $(1,0,0)$ plus a small perturbation. The timesteps are logarithmically spaced in the interval $[4\times10^{-6},4\times10^6]$ with 50 timesteps. The minimum and maximum timesteps are $3.03\times 10^{-6}$ and $1.72 \times 10^6$, respectively. The model is trained for 10,000 epochs with a learing rate of 0.001 decaying by a factor of 0.99 after each epoch.  In \cite{kim2021stiff} it is shown that training the imbalance in the magnitude of the components can lead to gradient pathologies. The trajectory for $y_2$ is rescaled
to $\hat{y}_2 = 10^4 y_2$ so that inputs and outputs of the network are $O(1)$. To recapture the original system all that is required is scaling $\hat{y_2}$ by $10^{-4}$.

Two NODEs are created to approximate the Robertson problem: one that utilizes the autoencoder and one that does not.  All other model and training parameters remain constant.
The nonlinear networks in the latent dynamics consist of 2 layers and 100 hidden dimensions.
The batch size is chosen to be 100. The learning rate is initially 0.01 and reduced by a factor of 0.99 at each step. The training is run for 10,000 epochs in each model. Lipschitz controlled layers are used to represent the nonlinear operator. In Figure~\ref{fig:RobertsonExample-a}, predictions for both models with initial condition $(1,0,0)$ are displayed. It is clear from this figure that the application of the autoencoder produces an accurate model for all components. 
Here the autoencoder acts like a coordinate transformation. {The latent space dynamics are presented in Figure \ref{fig:RobertsonExample-b}, illustrating how the non-trivial transformation induced by the autoencoder shifts the stiffness of the original system from the nonlinear operator into the linear operator. A non-trivial, stiff, latent dynamics is learned under a large temporal variation.}
We stress that both versions of the model are able to robustly integrate using timesteps ranging from $3.03\times 10^{-6}$ to $1.72 \times 10^6$ without rescaling time.
In both cases, the model is learning to extract a stiff linear operator that enables efficient integration on all scales.



\begin{figure}[htp]
    \centering
    \begin{subfigure}[t]{0.48\textwidth}
        \centering
        \includegraphics[width=\textwidth]{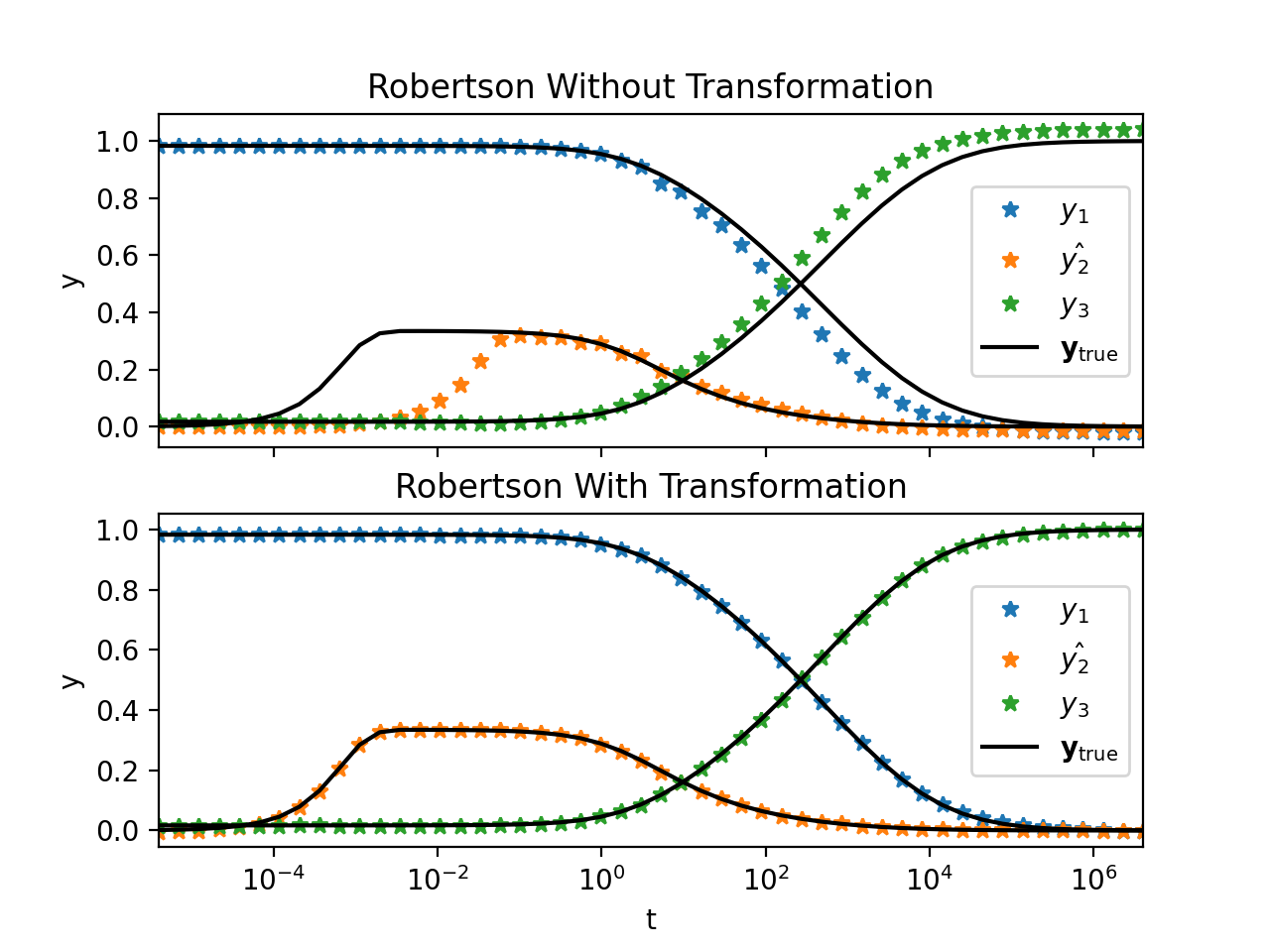}
        \caption{{Top: training without a transformation; Bottom: with an transformation. 
        Dotted lines are the true solution; colored solid lines are model predictions.}}
        \label{fig:RobertsonExample-a}
    \end{subfigure}\hfill
    \begin{subfigure}[t]{0.48\textwidth}
        \centering
        \includegraphics[width=\textwidth]{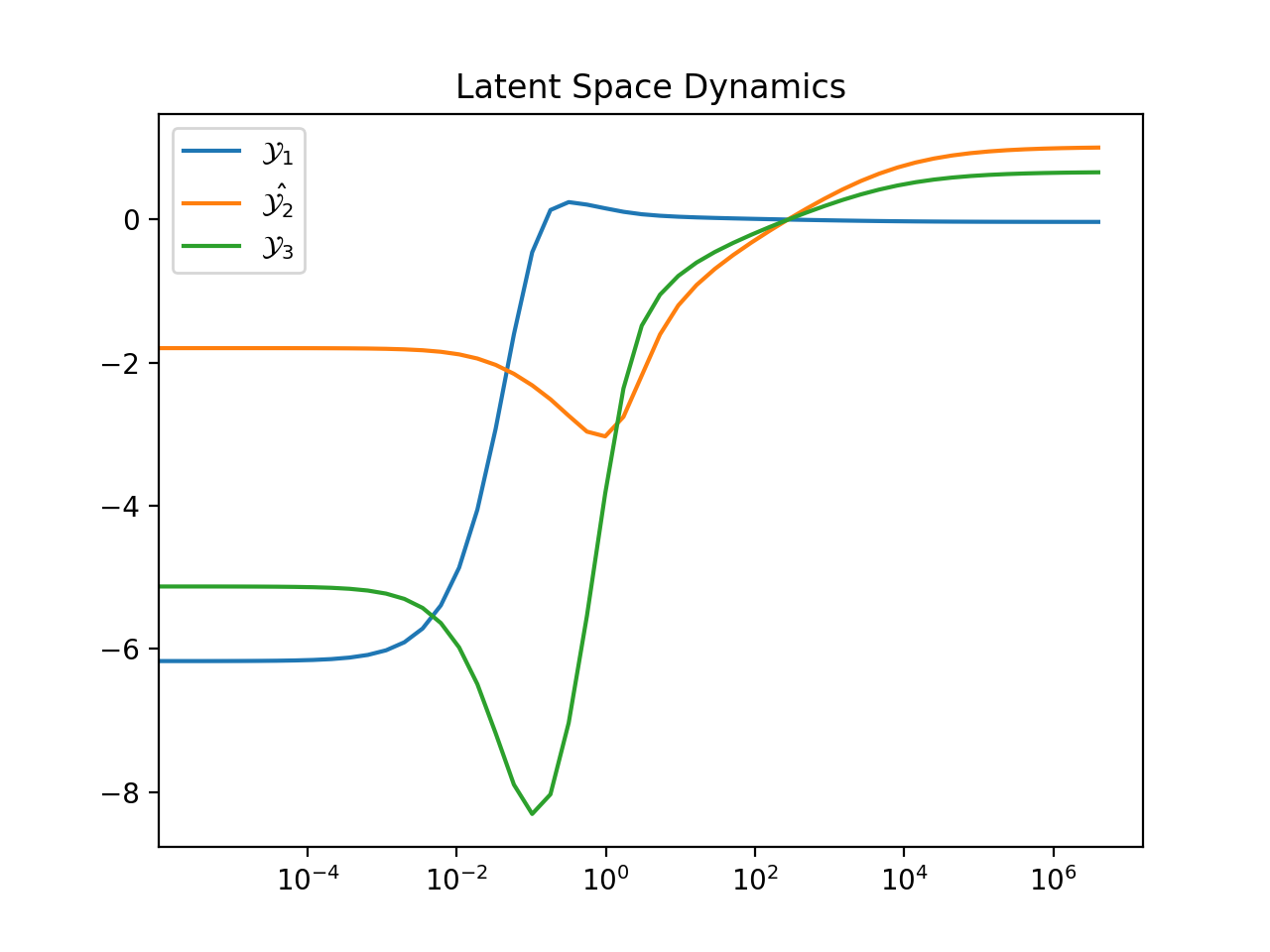}
        \caption{Latent space dynamics for Robertson's chemical reaction.}
        \label{fig:RobertsonExample-b}
    \end{subfigure}
    \caption{{(a) Prediction vs. truth with and without a transformation. (b) Latent dynamics corresponding to the same system when transformation is utilized.}}
    \label{fig:RobertsonCombined}
\end{figure}

Many existing approaches for learning stiff dynamics such as the Robertson problem rely on reparameterizing time and transforming the system into a non-autonomous system.
In \cite{caldana2024neural} the authors resolve the stiffness by learning a reparameterization in time. 
The resulting learned dynamical system is a nonstiff surrogate and may be cheaply evaluated using an explicit integrator. In data driven approaches for stiff dynamics \cite{xie2024latent, nockolds2025constant} a fixed time transformation is used, such as $\Delta \tilde{t} = -\frac{1}{\log{\Delta t}}$.
Our proposed approach avoids the need to know the time scale change a prior or learning the time reparameterization. 
The proposed approach can therefore bridge large variations in time scales--a desirable property for multi-scale problems that are often found in plasma physics and fluid dynamics.
{We observe that our approach performs better than \cite{caldana2024neural} with a much smaller error for the same test, particularly in the initial stage of the dynamics. \cite{caldana2024neural} obtains a worse result for this particular case due to its working on a rescaled dynamics, which does not prioritize the error in the true dynamics.}
The only other approach that we are aware of learning directly in the original time domain is \cite{chen2025due}, in which a multiscale structure is explicitly built in neural ODE. 

Another interesting finding from this example is that the proposed approach enables ETD for a system that might otherwise be considered unsuitable for it.
Specifically, one may not be able to use ETD to integrate the Robertson problem in its present form due to the stiffness being in the nonlinear part. 
Although an exponential integrator allows for larger timesteps than many standard explicit methods, a larger step can be only achieved if the variables are transformed so that the stiffness is transferred to the linear terms. This transformation is learned through data when we use an autoencoder to learn a favorable coordinate system. That is, two neural networks are applied, one to transform into the favorable coordinates for evolution and one to transform back into the original coordinates to compute the loss. 
These transformations along with 
restricting the linear operator to be Hurwitz stable are ingredients that enable robust training and, more importantly, efficient forward predictions for 
a challenging system such as Robertson.

\begin{figure}[ht]
{\color{blue}
    \centering    \includegraphics[width=.5\textwidth]{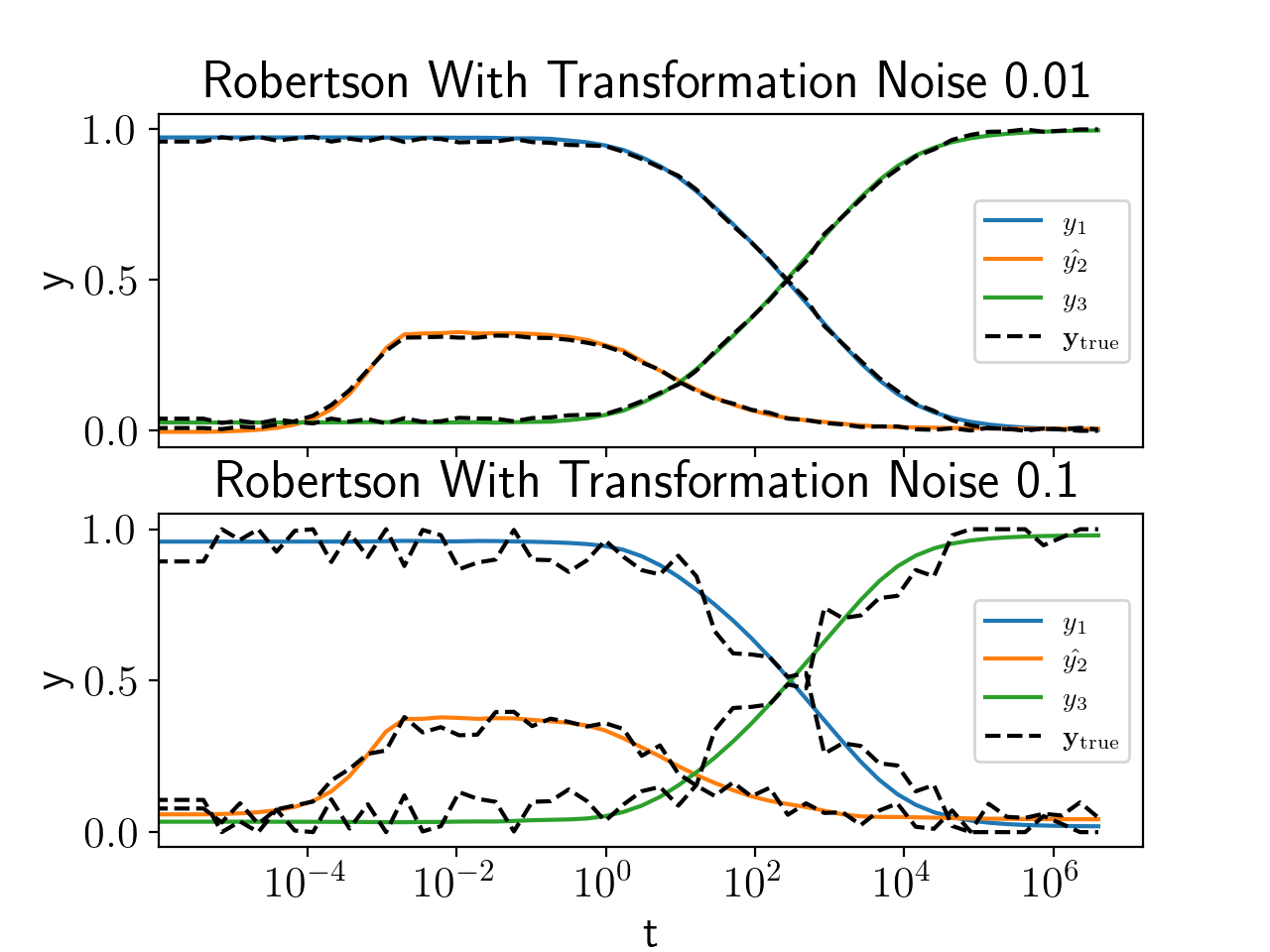}
    \caption{{Model and true trajectories with Gaussian noise scaled by 0.01 (top) and 0.1 (bottom).}}
    \label{fig:robersonNoise}
    }
\end{figure}

\subsubsection{Network Ablation Studies}

{
To evaluate the robustness of the proposed network, we perform two ablation studies: (1) assessing the model’s sensitivity to noisy inputs and (2) analyzing the effect of the Lipschitz constraint on the nonlinear operator.}
{
In the first study, Gaussian noise is added to each trajectory used for training, with noise levels scaled by factors of 0.01 and 0.1. As shown in Figure~\ref{fig:robersonNoise}, the model remains robust across all noise amplitudes, demonstrating strong resilience to perturbations in the training data.}
{
Table~\ref{tab:Lipschitz} presents the results of the second study, which investigates the impact of varying the Lipschitz constant on model performance. 
Three configurations are compared: no constraint, $L=1.0$, and $L=2.0$.
Additionally, the scaling factor applied to the initial weights of the nonlinear function is varied from $10^{-3}$ to $10^0$.
The results indicate that imposing a Lipschitz constraint significantly improves accuracy and stability across all initialization scales. In contrast, models trained without the constraint exhibit less consistent and occasionally unstable behavior. These findings also highlight that the performance remains sensitive to the magnitude of the initial weights, underscoring the importance of proper initialization when applying Lipschitz regularization. In each case the learning rate is initialized at 0.01 and has a decay rate of 0.9. The network has 100 hidden dimensions and 2 layers. A total of 20000 iterations are done and the final loss is reported in Table \ref{tab:Lipschitz}.}

\begin{table}[htb]
\centering
\caption{Effect of Lipschitz control constant and initilzation scale on model performance. 
Values represent the final training loss across different initialization scales.}
\label{tab:Lipschitz}
\vspace{0.5em}
\begin{tabular}{c|ccc}
\hline
\textbf{Init. Scale} & \textbf{No Lipschitz} & \textbf{$L = 1.0$} & \textbf{$L = 2.0$} \\
\hline
$10^{-3}$ & 0.052662 & 0.004060 & 0.004084 \\
$10^{-2}$ & 0.413053 & 0.004100 & 0.004087 \\
$10^{-1}$ & 0.004639 & 0.000201 & 0.000261 \\
$10^{0}$  & 0.023333 & 0.004182 & 0.004027 \\
\hline
\end{tabular}
\end{table}



\subsection{Kuramoto-Sivashinsky Equation}
\label{sec:ksEQN}
In this section we consider the one-dimensional Kuramoto-Sivashinsky (KS) equation 
\begin{equation}
\frac{du}{dt} = - u\frac{du}{dx} - \frac{d^2u}{dx^2} - \frac{d^4u}{dx^4},
\label{eq:ksEQN}
\end{equation}
with periodic boundary conditions on a domain of length $L = 22$. The KS equation exhibits chaotic behavior; therefore, we do not expect pointwise accuracy in the long term. To check the validity of our solution we rely on statistical measures, namely, the probability distribution function of the first and the second spatial derivatives, $u_x$ and $u_{xx}$.

We consider the dataset used in~\cite{linot2023stabilized}. 
Solutions are found by performing a Galerkin projection onto Fourier modes and using exponential time differencing to evolve the ODE forward in time \cite{kassam2005fourth}. 
After solving in Fourier space, we then transform the data back to physical space for training the neural ODEs. 
We sample the data at every 1.0 time units. The training data is constructed by separating the long trajectory into trajectories with 8 steps. Our batch size is 2000 and our learning rate is 0.001. The learning rate is reduced by a factor of 0.99 at each step. For this problem we use a Lipschitz controlled feed-forward neural network 
for the nonlinear part using 2 layers and inner dimension 200. 
We observe that Lipschitz control simultaneously
improves training performance 
and the long-term stability of the dynamic model.

There are a few additional steps taken to train this model effectively. First, a shift of the spectrum  to the Hurwitz matrix is applied during training. This is required since the KS system has both positive and negative eigenvalues and the Hurwitz parametrization requires the real part of the eigenvalues to be negative. Second, an autoencoder that projects the dynamics down to a latent space. There are many attractive options for model reduction. A popular method is utilizing projection based reduced-order modeling, which attempt to find a low-dimensional trial subspace that can represent the state space. Most of these trial subspaces are linear and can not capture the complicated dynamics of the KS equation. Moreover, the KS equation has a slow decaying Kolomogorov $n$-width \cite{mojgani2023kolmogorov}, which implies linear trial subspaces may not accurately represent the system. The power behind autoencoders is that they attempt to find an identity mapping with fewer dimensions. If the data can be represented in a lower dimensional nonlinear manifold, then optimization of the encoder and decoder
finds the optimal low-order representation. For this example 64 grid points are projected into a latent space of dimension 20 using an autoencoder. For $L=22$, the KS equation has 3 positive Lyapunov exponents, resulting in the Kaplan-York dimension of its attractor to be 5.198 \cite{edson2019lyapunov}. Therefore, it needs an inertial manifold of at least 6 dimensions. Our second model utilizes a latent space of 7 dimensions. 
This finding is consistent with the previous work \cite{linot2020deep}, where the authors found that for the exact same KS test, the minimal latent dimension for their model to train reasonably is 7.
We note that relative to our baseline architecture without an autoencoder, we observe factors of 12 and 24 speedup for the models with dimension 20 and 7, respectively. 

\begin{figure}[ht]
    \centering    \includegraphics[width=1.0\textwidth]{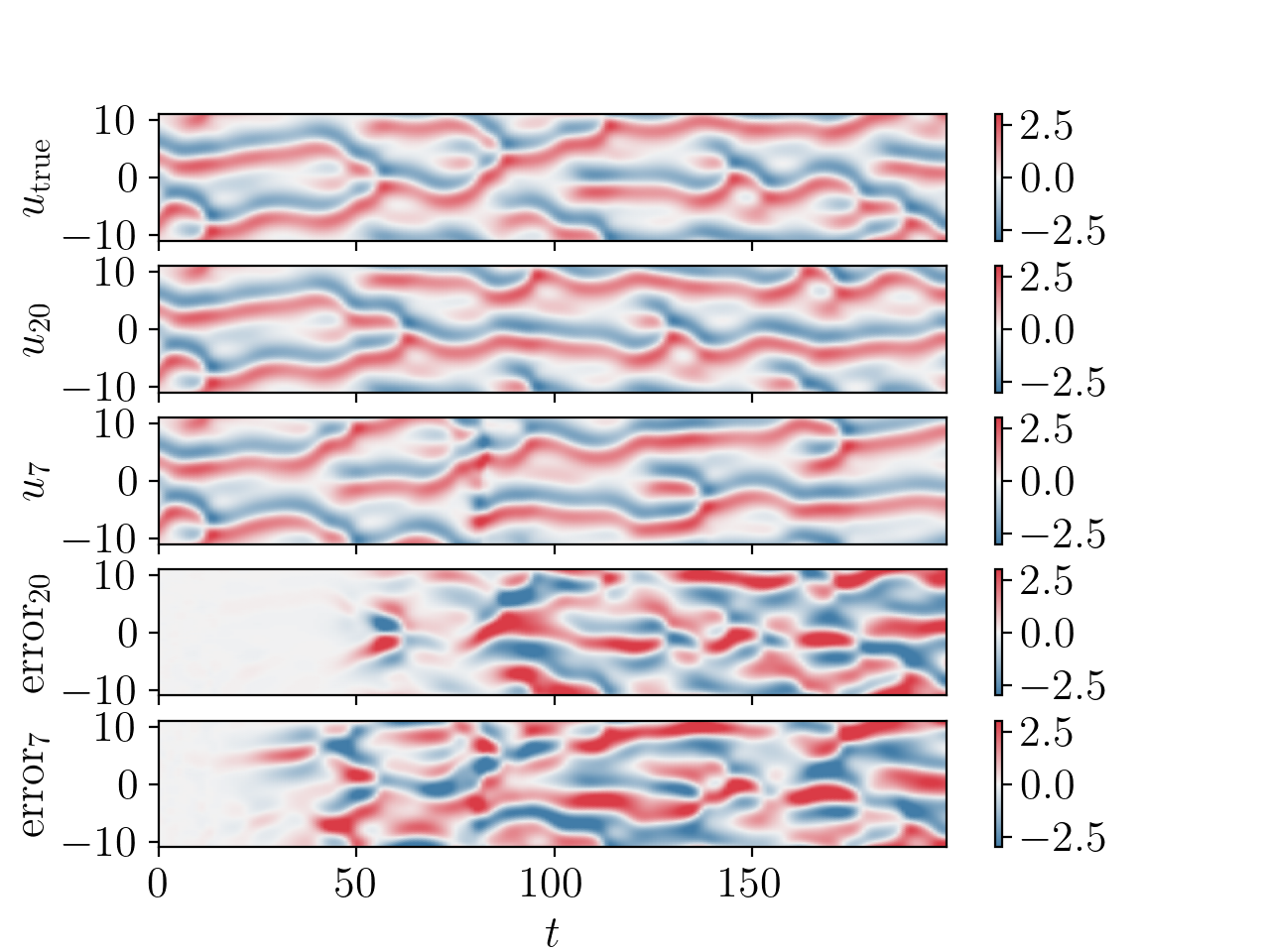}
    \caption{Displayed is the true solution ($u_{true}$), the model prediction with latent dimension 20 and 7 ($u_{20}$,$u_{7}$) and the errors ($\rm{error}_{20},\rm{error}_{7})$ .}
    \label{fig:ksResults}
\end{figure}

In Figure \ref{fig:ksResults} the evolution of the true solution ($U_{\text{true}}$), latent space dimension 20 model and latent space dimension 7 model  ($U_{20},U_7$), and the errors ($\text{error}_{20},\text{error}_7$) are displayed until time $t = 200$. The dynamics agree pointwise until around time $t = 50$ for latent space dimension 20 and until time $t = 25$ for latent space dimension 7. After the time of pointwise agreement we must rely on other measures to assess the validity of the model. 

\begin{figure}[htp]
    \centering
    \includegraphics[width=0.32\textwidth]{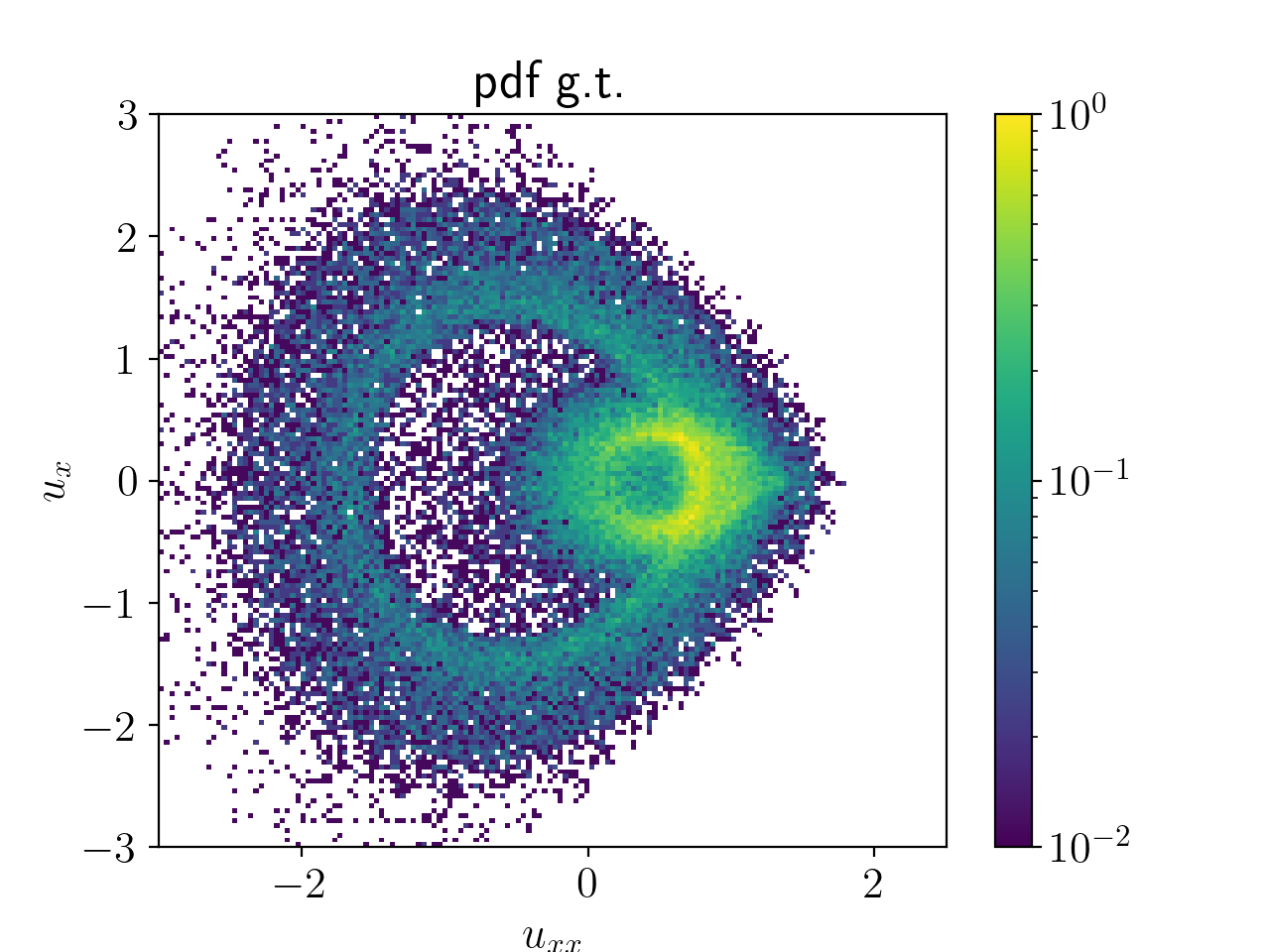}
    \includegraphics[width=0.32\textwidth]{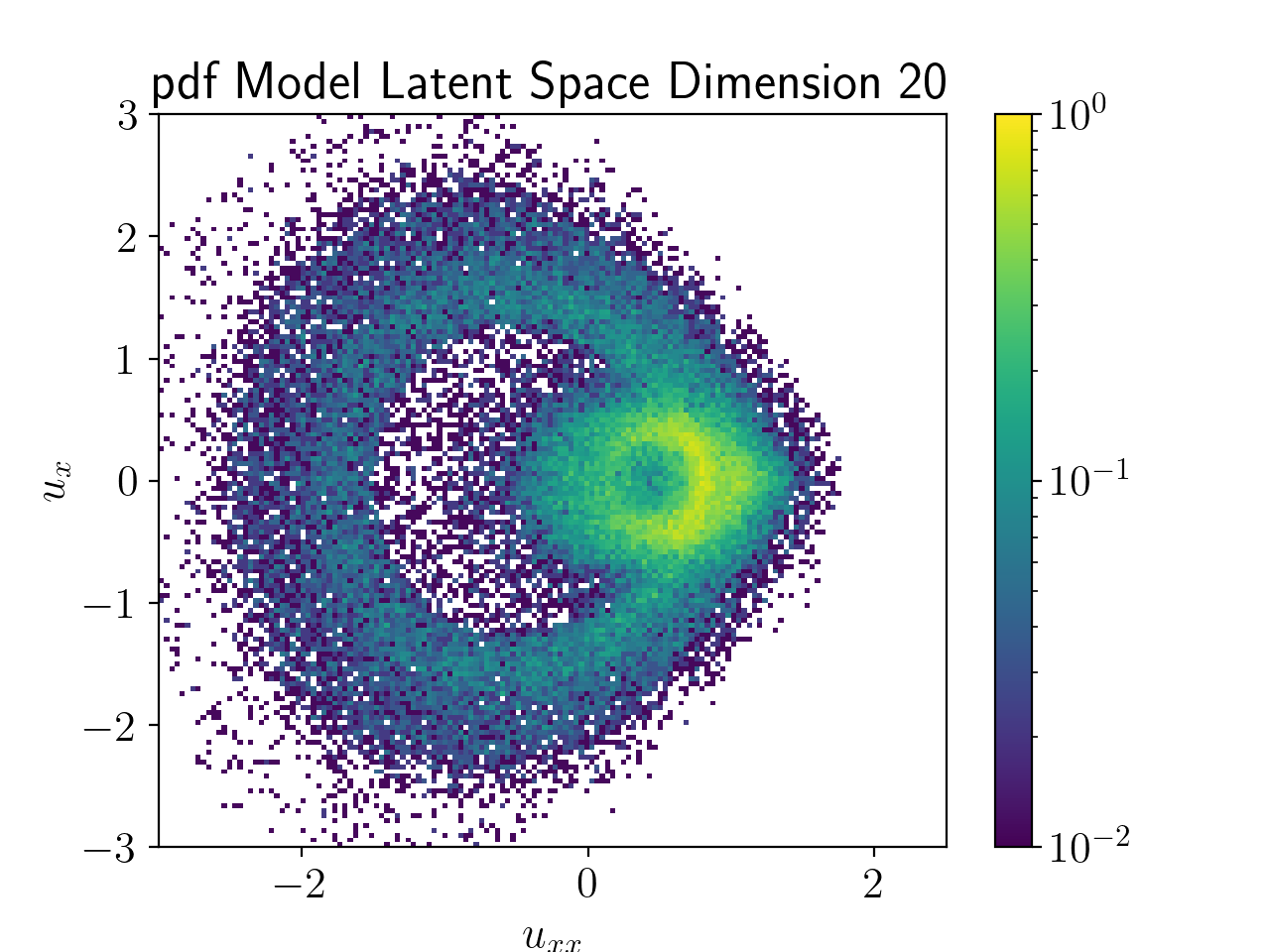}
    \includegraphics[width=0.32\textwidth]{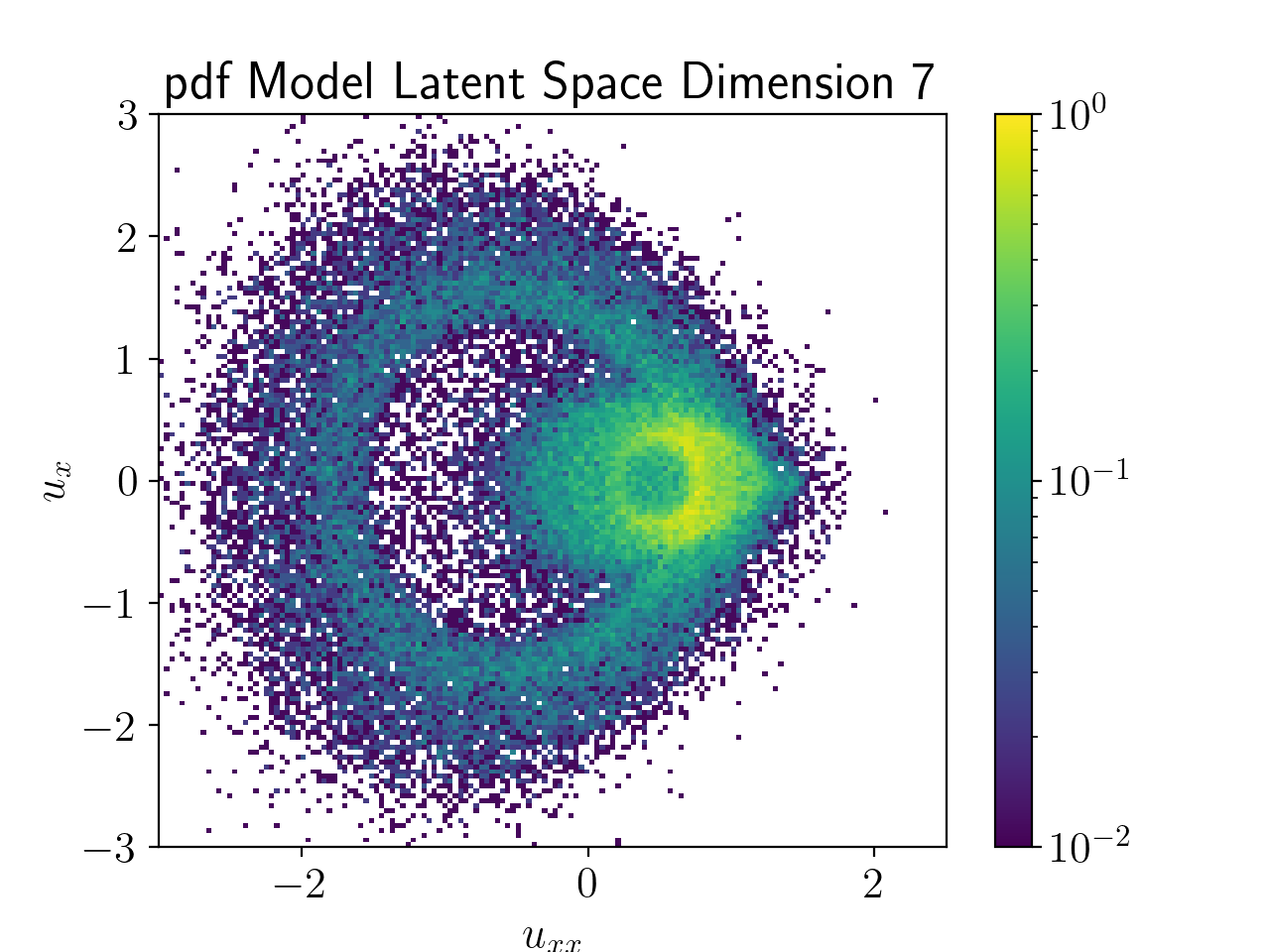}
    \caption{The joint probability distribution function of $u_x$ and $u_{xx}$ is displayed for the ground truth (Left) and the model with latent dimension 20 (Middle) and model with latent dimension 7 (Right).}
    \label{fig:ksResultsPDF}
\end{figure}

To verify the long-time behavior of predicting this chaotic system, Figure \ref{fig:ksResultsPDF} presents the joint probability distribution function of $u_x$ and $u_{xx}$. Qualitatively they match well demonstrating that the statistics of the dynamics agree between both models and the ground truth.

\section{Discussions}

A critical component of solving high-dimensional problems is dimensionality reduction. We pursue this through an autoencoder that compresses the state variables into a latent representation whose evolution satisfies an ordinary differential equation. Our approach can be viewed as an extension of latent neural ODEs~\cite{rubanova2019latent}, but with additional structure-preserving mechanisms built into the latent dynamics.
Latent neural ODEs learn a low-dimensional latent space in which continuous-time dynamics evolve, enabling the modeling of irregularly sampled or partially observed trajectories. Our framework extends this idea by introducing Hurwitz-stable linear operators, Lipschitz-controlled nonlinear components, and exponential integration, ensuring stable and efficient learning even for stiff or multiscale systems. This structure makes the method particularly suitable for physical systems where long-term accuracy and numerical robustness are essential.
Beyond latent ODEs, there are other extensions of the neural ODE framework designed to enhance representational capacity. Augmented neural ODEs~\cite{dupont2019} expand the state space with additional coordinates, allowing the model to represent more complex trajectories and overcome topological limitations of standard NODEs. In our examples, which involve first-order systems solved directly for the full state variable, augmentation is not physically motivated. However, the principles of exponential integration, Hurwitz stability, and Lipschitz control can naturally extend to augmented formulations. 
For PDEs in multiple spatial dimensions, the number of degrees of freedom in the discretization grows rapidly with the spatial dimension. To address such large-scale problems, it may be necessary to employ sparse or localized operators, such as convolutional, patch-based, or graph-based embeddings, to perform efficient coarse-graining before applying the learned structure-preserving neural ODE. Extending these ideas to high-dimensional PDEs represents an important step toward modeling realistic multi-scale physical phenomena, including fluid turbulence, plasma dynamics, and climate processes.

The proposed framework shows potential for extension to more complex dynamical systems beyond the examples considered in this work. In particular, it can naturally be applied to multi-scale systems, where components evolve on widely separated time scales. In our recent work on fast-slow dynamical systems~\cite{serino2024intelligent}, we demonstrated that a structure-preserving formulation can accurately capture the fast, dissipative dynamics through a linear-nonlinear split while maintaining stability and efficiency. The present approach builds on these ideas, and can be directly applied to the fast subsystem to achieve stable integration across multiple scales, while future work will target systems that exhibit strongly oscillatory or resonant behavior.
For non-smooth systems, such as those involving shocks, contact discontinuities, or material interfaces, the performance of the method will depend primarily on the spatial discretization employed. Nevertheless, the exponential time-stepper introduced in this paper can be readily combined with shock-capturing or discontinuity-aware spatial schemes, enabling robust integration even in the presence of non-smooth features.
We also note that hybrid dynamical systems, which involve discrete-continuous interactions or state-dependent switching, fall outside the continuous ODE framework considered here. Extending the current formulation to handle event-driven updates or discontinuous vector fields would require additional mechanisms, such as switching-aware integration or piecewise structure-preserving embeddings. Addressing these challenges represents a compelling direction for future research toward applying structure-preserving neural ODEs to a broader class of real-world, multi-physics problems.

\section{Conclusions}
\label{sec:conclusions}
This work presents a novel approach to addressing the long-term stability issues of neural ordinary differential equations (NODEs), a crucial aspect for efficient and robust modeling of dynamical systems. By combining a structure-preserving NODE with a linear and nonlinear split, an exponential integrator, constraints on the linear operator through Hurwitz matrix decomposition, and a Lipschitz-controlled neural network for the nonlinear 
operator,
the proposed approach demonstrates significant advantages in both learning and deployment over standard explicit and implicit NODE methods. This approach enables the efficient modeling of complex stiff systems and provides a stable foundation for deploying NODE models in real-world applications.
We demonstrate our approach in various examples. 
For the weakly nonlinear ODE example, we show that using an exponential integrator provides significant accuracy over a semi-implicit scheme.
We also demonstrate the structure-preserving NODE approach using stiff ODEs of the Robertson chemical reaction problem and PDEs of the Kuramoto-Sivashinksy system,
showcasing its ability to learn both multi-scale and 
chaotic systems.

\das{
While our experiments focus on canonical stiff systems, the proposed framework is general and can be applied to other domains. In particular, neural ODEs trained on image, speech, or sequential data may exhibit effective stiffness during optimization, making our structure-preserving approach potentially beneficial for broader deep learning tasks.
We plan to investigate this connection in future work by evaluating the method on standard benchmarks where stiffness may arise implicitly during training.
Beyond deep learning, the framework can also be extended to real-world scientific and engineering problems such as collisional-radiative modeling, climate and geophysical forecasting, and computational fluid dynamics, where stiffness and multi-scale interactions are prevalent. These applications represent natural testbeds for assessing the scalability, stability, and interpretability advantages of structure-preserving neural ODEs in realistic, high-dimensional environments.
}

\section*{Acknowledgments}
This research used resources provided by the National Energy Research Scientific Computing Center (NERSC), a U.S. Department of Energy Office of Science User Facility located at Lawrence Berkeley National Laboratory, operated under Contract No.~DE-AC02-05CH11231 using NERSC award ASCR-ERCAP0023112.
We thank Prof.~A.~J.~Roberts from  University of Adelaide on the suggestion related to the last example training.

\section*{Data Availability}

{The code and datasets analyzed during the current study are available from the corresponding author on reasonable request.}

\bibliographystyle{elsarticle-num}
\bibliography{references}
\end{document}